\documentclass[11pt]{amsart}
\usepackage{amsfonts}
\usepackage{amsmath}
\usepackage{amssymb}
\usepackage{graphicx}%

\usepackage{amscd}
\usepackage{mathrsfs}

\newtheorem{theorem}{Theorem}[section]
\newtheorem{lemma}[theorem]{Lemma}
\newtheorem{corollary}[theorem]{Corollary}
\newtheorem{definition}[theorem]{Definition}

\numberwithin{equation}{section}

\def\D{\mathscr D}
\def\DO{\mathcal D}
\def\RE{\mathbb R}
\def\CO{{\mathbb C}}
\def\N{\mathbb N}

\def\H{\mathcal H}
\def\F{\mathcal F}

\def\C{\mathcal C}
\def\L{\mathcal L}
\def\K{\mathcal K}
\def\A{\mathcal A}
\def\B{\mathcal B}
\def\P{\mathcal P}
\def\la{\lambda}

\def\max{\text{\rm max}}

\def\wh{\widehat}
\def\ol{\overline}

\begin{document}

\title[Singular potentials in one dimension]
{One-dimensional Schr\"odinger operators with singular potentials: A Schwartz distributional formulation}




\author{Nuno Costa Dias
\and Cristina Jorge
\and Jo\~{a}o Nuno Prata
}


\begin{abstract}
Using an extension of the H\"ormander product of distributions, we
obtain an intrinsic formulation of one-dimensional Schr\"odinger
operators with singular potentials. This formulation is entirely
defined in terms of standard {\it Schwartz} distributions and does
not require (as some previous approaches) the use of more general
distributions or generalized functions. We determine, in the new
formulation, the action and domain of the Schr\"odinger operators
with arbitrary singular boundary potentials. We also consider the
inverse problem, and obtain a general procedure for constructing
the singular (pseudo) potential that imposes a specific set of
(local) boundary conditions. This procedure is used to determine
the boundary operators for the complete four-parameter family of
one-dimensional Schr\"odinger operators with a point interaction.
Finally, the $\delta$ and $\delta'$ potentials are studied in
detail, and the corresponding Schr\"odinger operators are shown to
coincide with the norm resolvent limit of specific sequences of
Schr\"odinger operators with regular potentials.
\end{abstract}

{\maketitle }



\begin{section}{Introduction.}
Let $\widehat H_0=-D^2_x$ be the free Schr\"odinger operator with domain in the Sobolev-Hilbert space $\H^2(\RE)
\subset \L^2(\RE)$, and let $\widehat S$ be its symmetric restriction to the set $\D(\RE \backslash \{0\})$ of
smooth functions with support on a compact subset of $\RE \backslash \{0\}$. The self-adjoint (s.a.) extensions
of $\widehat S$ are usually called Schr\"odinger operators with point interactions
\cite{Albeverio1,Albeverio2,Berezin}. We shall denote them generically by $\widehat L$. In the one-dimensional
case they form a four-parameter family of operators, each of which is characterized by two boundary conditions
at $x=0$. These objects yield exactly solvable models
\cite{Albeverio1,Albeverio2,Albeverio3,Albeverio4,Albeverio6,Berezin,Brasche,Exner,Kostenko,Seba1} and have been
widely used in applications in quantum mechanics (e.g. in models of low-energy scattering
\cite{Albeverio7,Bolle1,Bolle2,Jensen} and quantum systems with boundaries
\cite{Dias5,Dias4,Dias3,Garbaczewski,Hirokawa}), condensed matter physics \cite{Avron,Cheon,Exner2} and, more
recently, on the approximation of thin quantum waveguides by quantum graphs
\cite{Albeverio0,Cacciapuoti1,Cacciapuoti2,Antonio}.

An interesting topic is the relation between the operators $\widehat L$ and the additive perturbations of the
operator $\widehat H_0$ by sharply localized potentials. The books \cite{Albeverio1,Albeverio2,Koshmanenko} and
the papers \cite{Albeverio4,Albeverio5,Berezin,Exner,Golovaty3,Kurasov1,Nizhnik2,Seba0,Seba2,Zolotaryuk2}
provide a detailed discussion and an extensive list of references on this subject. In the present work we will
address the problem, that emerges naturally in this context, of constructing a boundary potential formulation of
the operators $\widehat L$ that is entirely defined in terms of standard {\it Schwartz} distributions.

Let us consider this problem in more detail. The aim is to determine, for each $\widehat L$, a {\it boundary
potential operator} $\widehat B=B\cdot$, acting by multiplication by a {\it Schwartz} distribution $B$ (called
the {\it boundary potential}), and such that:
\begin{equation}
\widehat L=\widehat H_0+\widehat B
\end{equation}
where $\widehat H_0=-D_x^2$ is defined in the generalized sense. Since $\widehat L \psi =\widehat S \psi$ for
all $\psi \in \DO(\widehat S)=\D(\RE \backslash \{0\})$, the distribution $B$ has to be supported at $x=0$
(assuming that the product $\cdot$ is local). Hence, for $\psi \in \DO(\widehat L)$, the term $\widehat B \psi$
will in general stand for the product of a singular distribution by a discontinuous function. Such a product is
not well defined within the standard theory of Schwartz distributions and so, unless some additional structure
is introduced, the r.h.s. of (1.1) has only a formal meaning.

A common {\it interpretation} of the formal operator $\widehat H_0+\widehat B$ is that it stands for the limit
(e.g. in the norm resolvent sense) of sequences of operators of the form
\begin{equation}
\widehat H_n=\widehat H_0+\widehat B_n \quad , \quad \widehat B_n=B_n \cdot
\end{equation}
where $B_n$ belongs to some space of regular functions and
satisfies $B_n \longrightarrow B$ in $\D'$. Sequences of this kind
have been used in many applications, (see, for instance,
\cite{Golovaty3} and the references therein), and their
convergence properties were studied for particular cases
\cite{Albeverio0,Cacciapuoti1,Cacciapuoti2,Christiansen,Golovaty3,Seba0,Zolotaryuk0,Zolotaryuk2}.
Unfortunately, the relation between the convergence of $B_n$ in
$\D'$ and the convergence of the associated operators $\widehat
H_n$ in the norm resolvent is not straightforward
\cite{Brasche,Golovaty1,Golovaty2,Zolotaryuk2,Zolotaryuk1}. One
finds that different sequences of regular potentials, converging
to the same distribution, may lead to sequences of operators (1.2)
converging to distinct operators. The case $B=\delta'$ was studied
in detail. It was found that the formal operator $\widehat
H_0+\widehat\delta'$ stands (in the sense above) for, at least, a
one-parameter family of distinct s.a. Schr\"odinger operators with
a point interaction \cite{Golovaty3,Seba0,Zolotaryuk1}.

An interesting problem is then whether the operators $\widehat H_0+\widehat B$ admit an intrinsic formulation in
terms of distributions and also, whether such formulation coincides with the norm resolvent limit of suitable
sequences of Schr\"odinger operators with regular potentials. The first of these problems was addressed by P.
Kurasov \cite{Kurasov1} (and extended to higher order linear operators by Kurasov and Boman \cite{Kurasov2})
using a theory of distributions acting on discontinuous test functions (see also [chapter 3, \cite{Albeverio2}]
and \cite{Kurasov3,Nizhnik2}). Let us denote by $\widetilde\K$ the set of $\C^{\infty}(\RE \backslash
\{0\})$-functions that display (together with all their derivatives) finite lateral limits at $x=0$, and let
$\K$ be the set of $\widetilde\K$-functions with bounded support. The set of Kurasov distributions is the dual
of $\K$, denoted by $\K'$. As usual, one can introduce most of the algebraic operations in $\K'$. One can also
introduce a distributional derivative $D_\K$ and, more importantly, a product $\star_\K$ of an element of $\K'$
by an element of $\widetilde\K$. This product is defined for all $F \in \K'$ and $\psi\in \widetilde\K$ by:
$$
\langle F \star_\K \psi,\xi\rangle =\langle F,\psi\xi \rangle \quad , \quad \forall \xi\in \K
$$
where $\langle \, , \, \rangle$ denotes the distributional bracket. Using the product $\star_\K$ one may define
the {\it Kurasov} boundary operators $\widehat B_\K\psi=B_\K \star_\K \psi$ where $B_\K$ is a {\it Kurasov}
distribution. It turns out that by considering a simple extension of the product $\star_\K$, the complete four
parameter family of Schr\"odinger operators $\widehat L$ can be written in the form (1.1) with $\widehat B=B_\K
\star_{\K}$ and $B_\K$ a (pseudo) potential (since in general it also involves the operator $D_\K$) from $\K'$
\cite{Albeverio2,Kurasov1}.

Besides Kurasov's product, the Colombeau formalism \cite{Colombeau} has also been used \cite{Antonevich} to
obtain a precise definition of the boundary potential operators $\widehat B$. One should notice, however, that
both formalisms are defined outside from the space of Schwartz distributions, which renders the formulation of
the Schr\"odinger operators considerably more intricate than in the standard Schwartz case. For instance, the
Kurasov operators $\widehat H_\K=-D_\K^2+\widehat B_\K$ are written in terms of distributions in $\K'$ and the
new derivative operator $D_\K$. It turns out that $D_\K$ is a complicated object that does not satisfy the
standard properties of a derivative operator (it does not satisfy the Leibnitz rule, it does not reproduce the
derivative of functions for smooth regular distributions and the derivative of a constant is not zero
\cite{Kurasov1}). Moreover, Kurasov's formulation yields operators of the form $\widehat H_\K:\K \rightarrow
\K'$ and so $\widehat H_\K \psi \notin \D'$, not even for $\psi \in \D$. If one needs to produce standard
distributional results (like when modelling point interactions), in the end one has to project down to $\D'$ the
results of the formulation in $\K'$.

An alternative approach would be to use {\it intrinsic} products
of Schwartz distributions in the definition of $\widehat B$. Such
products have been used in the past to obtain consistent
formulations of ordinary and partial differential equations with
singular terms \cite{Hormann,Sarrico}. However, and up to our
knowledge, this approach has never been considered in the context
of Schr\"odinger operators with singular potentials.

The main goal of this paper is then to follow the latter approach and show that a rigorous formulation of the
operators $\widehat H_0+\widehat B$ can be obtained strictly within the framework of the Schwartz distributional
theory. Our approach rests upon the definition of an {\it intrinsic} multiplicative product of Schwartz
distributions that we shall denote by $*$. This product was originally defined and studied in \cite{Dias1}. It
constitutes a generalization (to the case of possibly intersecting singular supports) of the H\"ormander product
of distributions with non-intersecting singular supports (pag.55, \cite{Hormander}). The product $*$ is
associative, distributive, non-commutative and local. It satisfies the Leibnitz rule with respect to the
Schwartz distributional derivative $D_x$ and reproduces the usual product of functions for regular
distributions. Contrary to what happens in the Colombeau and Kurasov cases, the product $*$ is defined (and is
an inner operation) in a subspace of $\D'$. This subspace, endowed with the product $*$ and the Schwartz
distributional derivative $D_x$, becomes an associative differential algebra that contains Kurasov's function
space $\widetilde\K$ and, more importantly, all the distributional derivatives of the elements of
$\widetilde\K$.

Using the product $*$ we can define the {\it boundary potential
operators}
\begin{equation}\label{B0}
\widehat B \psi =  \psi * B_1 + B_2* \psi \,
\end{equation}
for an arbitrary pair $B_1,B_2$ of Schwartz distributions with
support on a finite set. The operators $\widehat B$ are {\it
extensions} of the operators $B \cdot$ where $B=B_1+B_2$ and
$\cdot$ is the standard product of a distribution by a test
function. In particular, they are well-defined on the spaces of
discontinuous functions that are important for studying point
interactions. Let us denote the association between the operators
$\widehat B$ (\ref{B0}) and the corresponding boundary potentials
$B=B_1+B_2$ by $ B \leftrightarrow  \widehat B$. Notice that
different choices of the pair $B_1,B_2$ (satisfying $B_1+B_2=B$)
yield different operators $\widehat B$, and so the association $B
\leftrightarrow \widehat B$ is not one-to-one.

More generally, we can also define the {\it boundary pseudo
potential operators}
\begin{equation} \label{BPP0}
\wh B \psi = \sum_{i,j=0}^1 D_x^{i} \wh B_{ij} D_x^{j} \psi
\end{equation}
where $\wh B_{ij}$ are boundary potential operators of the form
(\ref{B0}). The operators (\ref{BPP0}) will be use to construct a
boundary pseudo potential formulation of all s.a. Schr\"odinger
operators $\wh L$.

The operators of the form (\ref{B0}) and (\ref{BPP0}) are particular examples of {\it boundary operators}. We
use this terminology to denote an arbitrary linear operator $\wh B$ for which exists a finite set $I \subset
\RE$, such that
\begin{equation}\label{BOP}
\mbox{supp}\,\,\, \wh B \psi \subseteq I \, , \quad \forall \psi
\in \DO(\wh B) \subset \D' \, .
\end{equation}

In this paper we study several properties of the operators
$\widehat H_0 +\widehat B$, when $\wh B$ is an operator of the
form (\ref{B0}), (\ref{BPP0}) or (\ref{BOP}). In the next section
we review, for the convenience of the reader, the definition and
the main properties of the distributional product proposed in
\cite{Dias1}. In section 3, we prove that the operators $\widehat
B$, given by (\ref{B0}), can be equivalently defined as weak
operator limits of particular sequences of operators of
multiplication by smooth functions; these sequences converge in
$\D'$ to the associated boundary potential $B$. In section 4, we
prove some general results about the operators of the form $\wh
H_0 + \wh B$, when $\wh B$ is an arbitrary boundary operator
(\ref{BOP}). In particular, we determine their action and domain,
explicitly. In section 5, we develop a general method for finding
a boundary pseudo potential operator (\ref{BPP0}) that imposes a
specific given set of boundary conditions. This method is then
used to obtain a boundary pseudo potential formulation of {\it
all} one-dimensional Schr\"odinger operators $\widehat L$. We also
show that such a representation is not possible (for all $\wh L$)
if only boundary potential operators of the form (\ref{B0}) are
used. Finally, in section 6, we study the operators (\ref{B0}) in
detail. For an arbitrary boundary potential $B=a \delta(x)+b
\delta'(x)$, we determine which operators $\widehat H_0+\widehat
B$, where $\widehat B \leftrightarrow B$, are s.a. and conversely,
which s.a. operators $\wh L$ admit a boundary potential
representation. The particular cases $B=a\delta(x)$ and $B= b
\delta'(x)$, $a,b \in \RE$, are then studied in detail. We show
that they yield families of Schr\"odinger operators that coincide
(exactly in the first case, and to a large extend in the second
case) with the families of norm resolvent limit operators that
were obtained in \cite{Golovaty2,Golovaty3}.

\bigskip

\noindent\textbf{Notation}. Operators are denoted by letters with
a hat and distributions by capital roman letters ($F$, $G$...).
The exception is the Dirac measure $\delta$. Generic functions are
denoted by lower case Greek letters ($\psi$, $\phi$, $\xi$...).
Lower case roman letters from the middle of the alphabet ($f$,
$g$, $h$...) are reserved for continuous functions and those from
the end of the alphabet ($s$, $t$, $u$...) for smooth functions of
compact support. Capital Greek letters ($\Omega$, $\Xi$...) denote
open subsets of $\RE$. The functional spaces are denoted by
calligraphic capital letters ($\L^2(\Omega)$, $\H^2(\Omega)$,
$\D(\Omega)$...). When $\Omega = \RE$ we usually write only
$\L^2$, $\H^2$, $\D$, etc. The domain of an operator $\widehat A$
is written $\DO(\widehat A)$ and the statement $\widehat A
\subseteq \widehat B$ means, as usual, that $\DO(\widehat A)
\subseteq \DO(\widehat B)$ and $\widehat A \psi= \widehat B \psi$
for all $\psi \in \DO(\widehat A)$.
\end{section}

\begin{section}{A multiplicative product of Schwartz distributions}

In this section we review the main properties of the
multiplicative product of distributions $*$ that was proposed in
\cite{Dias1}. For details and proofs the reader should refer to
\cite{Dias1}.

We start by presenting some basic definitions. Let the $n$th order
singular support of a distribution $F \in \D'$ (denoted $n$-sing
supp $F$) be the closed set of points where $F$ is not a
$\C^n$-function. More precisely: let $\Omega \subseteq \RE$ be the
largest open set for which there is $f \in \C^n(\Omega)$ such that
$F|_\Omega=f$ (where $F|_{\Omega}$ denotes the restriction of $F$
to $\D(\Omega)$). Then $n$-sing supp $F=\RE \backslash \Omega$.
This definition generalizes the usual definition of singular
support of a distribution. We have, of course, $m$-sing supp $F
\subseteq n$-sing supp $F$ for all $m \le n$ and $\infty$-sing
supp $F=$ sing supp $F$.

Another useful concept is the order of a distribution
\cite{Kanwal}: we say that $F \in \D'$ is of order $n$ (and write
$n=$ ord $F$) iff $F$ is the $n$th order distributional derivative
(but not a lower order distributional derivative) of a regular
distribution.

Let now $\C_p^{n}$ be the space of piecewise $n$ times continuously differentiable functions: $\psi \in
\C_p^{n}$ iff there is a finite set $I\subset \RE$ such that $\psi \in \C^{n}(\RE \backslash I)$ and the lateral
limits $\lim_{x \to x_0^{\pm}} \psi^{(j)} (x) $ exist and are finite for all $x_0 \in I$ and all $j$-order
derivatives of $\psi$, with $j=0,..,n$.

Finally, let $\F^n$ be the space of distributions $F \in \D'$ such
that supp $F$ is a finite set and ord $F \le n+1$.

A distributional extension of the sets $\C_p^n$ is then given by:

\begin{definition}
Let $\A^{n}=\C_p^n \oplus \F^n$, where the elements of $\C_p^n$
are regarded as distributions. Moreover, the space of
distributions of the form $F|_{\Omega}$, where $F \in \A^n$, is
denoted by $\A^n(\Omega)$.
\end{definition}

We remark that only the spaces $\A^0$ and $\A^1$ (and their
restrictions $\A^0(\Omega)$ and $\A^1(\Omega)$) will be used in
the rest of the paper (sections 3 to 6). Hence, for the remainder
of this and the next sections, the reader may always assume that
$n=0,1$.

Let us then consider the definition of $\A^n$. We have, of course,
$\C_p^{n} \subset \A^{n} \subset \D'$. All the elements of $\A^n$
are distributions of order (at most) $n+1$ and finite $n$th order
singular support. They can be written in the form $F=\Delta_F +f$,
where $\Delta_F \in \F^n$ and $f \in \C_p^n$. The next Theorem
states this property in a more explicit form:

\begin{theorem}
$F \in \A^{n}$ iff there is a finite set $I=\{x_1,...,x_m\}$ associated with a set of open intervals
$\Omega_i=(x_i,x_{i+1})$, $i=0,..,m$ (where $x_0=-\infty$ and $x_{m+1}=+\infty$) such that ($\chi_{\Omega_i}$ is
the characteristic function of $\Omega_i$):
\begin{equation}
F= \sum_{i=1}^m \sum_{j=0}^n c_{ij}\delta^{(j)}(x-x_i) + \sum_{i=0}^m f_i \chi_{\Omega_i}
\end{equation}
for some $c_{ij} \in \CO$ and $f_i \in \C^{n}$. We have,
necessarily, $n$-sing supp $F \subseteq I$.
\end{theorem}

The product $*$ will be defined in the spaces $\A^n$. First, we
recall some basic definitions about products of distributions. Let
$\Xi \subseteq \RE$ be an open set. The standard product of $F \in
\D'(\Xi)$ by $g \in \C^\infty(\Xi)$ is defined by
$$
\langle Fg, t \rangle= \langle F, gt \rangle \quad , \quad \forall
t \in \D(\Xi) \, .
$$
Moreover, this product is also well-defined for all pairs $F \in
\A^n(\Xi)$, $g \in \C^n(\Xi)$, where $n\in \N_0$ (because ord $F
\le n+1$ and $gt \in \C^n_0(\Xi)$).

The H\"ormander product of distributions generalizes the previous
product to the case of two distributions with finite and disjoint
singular supports (pag.55, \cite{Hormander}). We present a
slightly more general definition which we call the {\it extended
H\"ormander product}.

\begin{definition}
For $n \in \N_0 $, let $F,G \in \A^n$ be two distributions such
that $n$-sing supp $F$ and $n$-sing supp $G$ are finite disjoint
sets. Then there exists a finite open cover of $\RE$ (denote it by
$\{\Xi_i \subset \RE,\, i=1,..,d \}$) such that, on each open set
$\Xi_i$, either $F$ or $G$ is a $\C^{n}(\Xi_i)$-function. Hence,
on each $\Xi_i$, the two distributions can be multiplied using the
previous product of a $\A^n$-distribution by a $\C^n$-function.
The extended H\"ormander product of $F$ by $G$ is then defined as
the unique distribution $F \cdot G \in \A^n$ that satisfies:
$$
F \cdot G|_{\Xi_i}= F|_{\Xi_i} G|_{\Xi_i} \quad , \quad  i=1,..,d.
$$
\end{definition}

Finally, the new product $*$ generalizes the extended H\"ormander
product to the case of an arbitrary pair of distributions in
$\A^n$:

\begin{definition}
The multiplicative product $*$ is defined for all $F,G \in \A^{n}$
by:
$$
F * G= \lim_{\epsilon \downarrow 0} F(x) \cdot G(x+\epsilon),
$$
where the product in $F(x) \cdot G(x+\epsilon)$ is the extended
H\"ormander product and the limit is taken in the distributional
sense.
\end{definition}

The explicit form of $F*G$ is given in Theorem 2.5 below and the
main properties of $*$ are stated in Theorem 2.6. Let $F,G \in
\A^{n}$ and let $I_F$ and $I_G$ be the $n$-singular supports of
$F$ and $G$, respectively. Let $I=I_F \cup I_G$ and write
explicitly $I=\{x_1,..,x_m\}$ (assuming $x_i < x_{i+1}$). Define
the open sets $\Omega_i=(x_i,x_{i+1})$, $i=0,..,m$ (with
$x_0=-\infty$ and $x_{m+1}=+\infty$). Then, in view of Theorem
2.2, $F$ and $G$ can be written in the form:
\begin{eqnarray}
F &=& \sum_{i=1}^m \sum_{j=0}^n a_{ij}\delta^{(j)}(x-x_i) + \sum_{i=0}^m f_i \chi_{\Omega_i} \nonumber \\
G &=& \sum_{i=1}^m \sum_{j=0}^n b_{ij}\delta^{(j)}(x-x_i) + \sum_{i=0}^m g_i \chi_{\Omega_i}
\end{eqnarray}
where $f_i,g_i \in \C^n(\RE)$ and $a_{ij}=0$ if $x_i \notin I_F$ or if $j \ge$ ord $F$, and likewise for $G$.
Then we have:
\begin{theorem}
Let $F,G \in \A^{n}$ be written in the form (2.2). Then $F*G$ is given explicitly by
$$
F * G = \sum_{i=1}^m \sum_{j=0}^n \left[ a_{ij} g_i (x) + b_{ij}
f_{i-1}(x) \right] \cdot \delta^{(j)}(x-x_i) + \sum_{i=0}^m f_i
g_i \chi_{\Omega_i}
$$
and $F*G \in \A^{n}$.
\end{theorem}

The main properties of $*$ are summarized in the following

\begin{theorem}
The product $*$ is an inner operation in $\A^{n}$, it is
associative, distributive, non-commutative and it reproduces the
extended H\"ormander product of distributions if the $n$-singular
supports of $F$ and $G$ do not intersect. Since $\A^{n}$ is not
closed with respect to the distributional derivative $D_x$, the
Leibnitz rule is valid only under the condition that $D_x F, D_xG
\in \A^{n}$.
\end{theorem}

\end{section}

\begin{section}{The Shifting Delta operators}

Using the spaces $\A^n$ and the product $*$, we can now define the
following sets of operators:

\begin{definition} \label{DefOp}
(1) Let $n \in \N_0$. The set of all {\it boundary potential
operators} of the form (\ref{B0}):
$$
\wh B: \, \DO(\wh B) \subset \D' \longrightarrow \D'; \quad \wh B
\psi = \psi * B_1 + B_2 * \psi
$$
where $B_1,B_2 \in \A^n$ and $B_1,B_2$ have finite support, is
denoted by $\wh \A^n$. We notice that $\cup_{i \ge n} \A^i
\subseteq \DO(\wh B)$. Moreover, $\wh\A^n
\subset \wh\A^m$ for all $n < m$.\\
\\
(2) The set of all {\it boundary pseudo potential operators}
(\ref{BPP0}) of the form:
\begin{equation}\label{BPP1}
\wh B: \, \DO(\wh B) \subset \D' \longrightarrow \D'; \quad \wh B
\psi = \wh B_1 \psi + \wh B_2 D_x \psi + D_x \wh B_3 D_x \psi
\end{equation}
where $\wh B_1 \in \wh\A^1$ and $\wh B_2,\wh B_3 \in \wh A^0$, is
denoted by $\wh\P$. We notice that $\cup_{i\ge 1} \A^i \subseteq
\DO(\wh B)$.\\
\\
(3) The linear operators $\wh B: \, \DO(\wh B) \subset \D'
\rightarrow \D'$ that satisfy (\ref{BOP}) are called {\it boundary
operators}. The set of all {\it boundary operators} is denoted by
$\wh\B$.

\end{definition}

We have, of course, $\wh\A^0 \subset \wh\A^1 \subset \wh\P \subset
\wh\B$. In this section, we will study two important families of
boundary potential operators:

\begin{definition}\label{Defdelta+}
For $n \in \N_0$ and $x_0 \in \RE$ let:
$$
\widehat\delta^{(n)}_+(x_0): \A^{n} \longrightarrow \A^{n}; \quad F \longrightarrow
\widehat\delta^{(n)}_+(x_0)F= \delta^{(n)}(x-x_0)*F
$$
$$
\widehat\delta^{(n)}_-(x_0): \A^{n} \longrightarrow \A^{n}; \quad F \longrightarrow
\widehat\delta^{(n)}_-(x_0)F= F * \delta^{(n)}(x-x_0) \, .
$$
We call $\widehat\delta^{(n)}_+(x_0)$ the $n${\it th-order right shifting delta}, and
$\widehat\delta^{(n)}_-(x_0)$ the $n${\it th-order left shifting delta}. For $n=0$, we simply write
$\widehat\delta_+(x_0)$ and $\widehat\delta_-(x_0)$; for $x_0=0$ we write $\widehat\delta^{(n)}_+$ and
$\widehat\delta^{(n)}_-$.
\end{definition}

We have $\wh\delta^{(n)}_\pm(x_0) \in \wh\A^n$. We also notice
that, since $\delta^{(n)}(x-x_0) \in \A^{m}$ for all $m \ge n$,
the operators $\widehat\delta^{(n)}_\pm(x_0)$ can be extended to
$\A^{m}$, for all $m \ge n$.

These operators satisfy the following basic properties:
\newline

(1) Let us define
$$
\widehat\delta^{(n)}(x_0): \C^n \longrightarrow \D'; \quad f
\longrightarrow \widehat\delta^{(n)}(x_0) f=
\delta^{(n)}(x-x_0)\cdot f
$$
where $\cdot$ is the extended H\"ormander product. The operator
$\widehat\delta^{(n)}(x_0)$ is just the standard operator of
"multiplication by the $n$-order derivative of a Dirac delta". We
have explicitly
\begin{equation}
\langle \delta^{(n)}(x-x_0)\cdot f,t \rangle = (-1)^n
\frac{d^n}{dx^n}(f \, t)(x_0) \quad , \quad \forall t \in \D \, .
\label{StandProd}
\end{equation}
Since, for all $f \in \C^n$ (cf. Theorem 2.5),
$$
\delta^{(n)}(x-x_0)*f=f*\delta^{(n)}(x-x_0)=\delta^{(n)}(x-x_0)\cdot f
$$
we have
$$
\widehat\delta^{(n)}_-(x_0) f=\widehat\delta^{(n)}_+(x_0) f =
\widehat\delta^{(n)}(x_0) f \quad , \quad \forall f \in \C^n
$$
and so the operators $\widehat\delta^{(n)}_+(x_0)$ and $\widehat\delta^{(n)}_-(x_0)$ are extensions of
$\widehat\delta^{(n)}(x_0)$ to the space $\A^n$.
\newline

(2) For all $n \in \N_0$, both $\widehat\delta^{(n)}_{-}(x_0)$ and $\widehat\delta^{(n)}_+(x_0)$ can be cast as
weak operator limits of one-parameter families of (operators of multiplication by) smooth functions. Let us set
$x_0=0$. For every $\epsilon
>0$, let $v_{\epsilon} \in \D$ be a non-negative and even function with support on $[-\epsilon,\epsilon]$ and such
that $\int_{-\infty}^{+\infty}  v_{\epsilon}(x)\, dx=1$. Since
$$
\lim_{\epsilon \downarrow 0} \langle v_{\epsilon},t \rangle =t(0)
\quad , \quad \forall t \in \D
$$
we have, in the sense of distributions, $\lim_{\epsilon \downarrow 0} v_{\epsilon}(x) = \delta(x)$.

Now consider the operators ($n \in \N_0$, $\epsilon >0$):
\begin{equation}
\widehat v_{\epsilon}^{(n)} : \A^{n} \longrightarrow \A^{n}; \quad F(x) \longrightarrow \widehat
v_{\epsilon}^{(n)} F(x)= v_{\epsilon}^{(n)}(x-\epsilon) \cdot F(x) \label{eqB}
\end{equation}
where $\cdot$ is the extended H\"ormander product. In the
distributional sense, we have once again
$$
\lim_{\epsilon \downarrow 0} v_{\epsilon}^{(n)}(x-\epsilon) = \delta^{(n)}(x)
$$
for all $n \in \N_0$. On the other hand, in the operator sense, we get:

\begin{theorem}
For all $n \in \N_0$, the one parameter family $\left(\widehat v_{\epsilon}^{(n)}\right)_{(\epsilon >0)}$
converges, in the weak operator sense, to the operator $\widehat\delta_+^{(n)}$, i.e.
$$
w-\lim_{\epsilon \downarrow 0} \widehat v_{\epsilon}^{(n)} = \widehat\delta_+^{(n)}.
$$
\end{theorem}

\begin{proof}
Let us start by considering the case $n=0$.  Let $F \in \A^{0}$
and $t\in \D$. We have
\begin{align}
\lim_{\epsilon \downarrow 0} \langle\widehat v_{\epsilon} F,t\rangle &=\lim_{\epsilon \downarrow 0} \langle v_{\epsilon}(x-\epsilon)\cdot F(x),t(x)\rangle \nonumber \\
&=\lim_{\epsilon \downarrow 0} \langle F(x), v_{\epsilon}(x-\epsilon) t(x) \rangle   \label{eq3.1} \\
&=\lim_{\epsilon \downarrow 0} \int_{0}^{2\epsilon} v_{\epsilon}(x-\epsilon)f(x)t(x) \, dx \nonumber
\end{align}
where we used the fact that $v_{\epsilon}(x-\epsilon)=0$ for all $
x \notin (0,2\epsilon)$ and also that, for sufficiently small
$\epsilon > 0$, there is a function $f\in \C^0$ such that $\left.
F\right|_{(0,2\epsilon)}=\left.f\right|_{(0,2\epsilon)}$ (cf.
Theorem 2.2).

Setting $g=ft \in \C^0$, we get
\begin{align*}
& \lim_{\epsilon \downarrow 0} \int_{0}^{2\epsilon} v_{\epsilon}(x-\epsilon)g(x) \, dx \\
= &\lim_{\epsilon \downarrow 0} \left[\int_{0}^{2\epsilon} v_{\epsilon}(x-\epsilon)g(0) \, dx
+\int_{0}^{2\epsilon} v_{\epsilon}(x-\epsilon)\left[g(x)-g(0)\right] \, dx \right] \\
=& g(0) + \lim_{\epsilon \downarrow 0} \int_{0}^{2\epsilon} v_{\epsilon}(x-\epsilon)\left[g(x)-g(0)\right] \, dx \, .
\end{align*}
Now, for every $\epsilon > 0$, let
$$
M_\epsilon=\underset{x \in [0,2\epsilon]}{\text{max}}  \, \left[g(x)-g(0)\right] \quad \text{and} \quad
m_\epsilon=\underset{x \in [0,2\epsilon]}{\text{min}}  \, \left[g(x)-g(0)\right] \, .
$$
Since $v_\epsilon(x-\epsilon)$ is non-negative, we have for all $\epsilon > 0$
$$
m_\epsilon \le \int_{0}^{2\epsilon} v_{\epsilon}(x-\epsilon)\left[g(x)-g(0)\right] \, dx \le M_\epsilon
$$
and since $\lim_{\epsilon \downarrow 0}M_\epsilon = \lim_{\epsilon \downarrow 0}m_\epsilon =0$, we get
$$
\lim_{\epsilon \downarrow 0} \int_{0}^{2\epsilon} v_{\epsilon}(x-\epsilon)\left[g(x)-g(0)\right] \, dx =0 \, .
$$
Hence
$$
\lim_{\epsilon \downarrow 0} \langle\widehat v_{\epsilon} F,t\rangle = g(0)=f(0)t(0) \, .
$$

The generalization of this result to $n \in \N$ is easily
obtained. Notice first that the equation (\ref{eq3.1}) is also
valid if we replace $\widehat v_{\epsilon}$ by $\widehat
v_{\epsilon}^{(n)}$ and require that $F\in \A^n$ (in which case
$f,g \in \C^n$). Integrating by parts, it follows that
$$
\lim_{\epsilon \downarrow 0} \int_{0}^{2\epsilon} v_{\epsilon}^{(n)}(x-\epsilon)g(x) \, dx
= (-1)^n \lim_{\epsilon \downarrow 0} \int_{0}^{2\epsilon} v_{\epsilon}(x-\epsilon)g^{(n)}(x) \, dx
=(-1)^n g^{(n)}(0) \, .
$$

Finally, from Definition \ref{Defdelta+} and Theorem 2.5, we also
have:
$$
\widehat\delta_+^{(n)}  F=\delta^{(n)}(x) \cdot f(x)
\Longrightarrow \langle \widehat\delta_+^{(n)} F,t\rangle = (-1)^n
\frac{d^n}{d x^n}(ft)(0) , \quad \forall t \in \D
$$
and so $w-\lim_{\epsilon \downarrow 0} \widehat v_{\epsilon}^{(n)} = \widehat\delta_+^{(n)} $, for all $n \in \N_0$.

\end{proof}

If we re-define
$$
\widehat v_{\epsilon}^{(n)} : \A^{n} \longrightarrow \A^{n}; \quad F(x) \longrightarrow \widehat
v_{\epsilon}^{(n)} F= v_{\epsilon}^{(n)}(x+\epsilon) \cdot F(x)
$$
then we also have
$$
w-\lim_{\epsilon \downarrow 0} \widehat v_{\epsilon}^{(n)} = \widehat\delta_-^{(n)}
$$
for all $n \in \N_0$.

Moreover, if we combine $v_{\epsilon}^{(n)}(x+\epsilon)$ and
$v_{\epsilon}^{(n)}(x-\epsilon)$, we easily obtain one-parameter
families of smooth functions that converge in $\D'$ to
$\delta^{(n)}(x)$
$$
c\, v^{(n)}_\epsilon(x-\epsilon) + (1-c)
v^{(n)}_\epsilon(x+\epsilon) \overset{\D'}{\longrightarrow}
\delta^{(n)}(x) \quad , \quad c \in \RE
$$
and which in the weak operator sense converge to
$$
w-\lim_{\epsilon \downarrow 0} \left[c
v^{(n)}_\epsilon(x-\epsilon)\cdot + (1-c)
v^{(n)}_\epsilon(x+\epsilon)\cdot\right]  =  c
\widehat\delta_+^{(n)} +(1-c) \widehat\delta_-^{(n)} \, .
$$

Finally, we notice that every boundary potential operator
$\widehat B \in \wh\A^n$ is given by a linear combination of the
operators $\widehat\delta_+^{(m)}$ and $\widehat\delta_-^{(m)}$,
$m \le n$. Hence, every $\widehat B \in \wh\A^n$ can also be
written as the weak operator limit of sequences of operators of
multiplication by smooth functions. The associated sequence of
smooth functions converges in $\D'$ to $B=B_1+B_2 \leftrightarrow
\widehat B$.
\newline

(3) Since the Sobolev-Hilbert spaces $\H^2(\RE_{\pm})$ satisfy
$$
\H^2(\RE_{\pm}) = \chi_{\RE_{\pm}} \left[\H^2\right]\quad
\mbox{and} \quad \H^2 \subset \C^1,
$$
we have $\H^2(\RE_-) \oplus \H^2(\RE_+) \subset \C_p^1 \subset
\A^{1}$ and $D_x[ \H^2(\RE_-) \oplus \H^2(\RE_+)] \subset \A^{0}$.
Hence, the three operators $\widehat\delta_{\pm}(x)$,
$\widehat\delta'_{\pm}(x)$ and $\widehat\delta_{\pm}(x)D_x$ are
well defined on $\H^2(\RE_-) \oplus \H^2(\RE_+)$. The same is then
true for all $\wh B \in \wh\P$.

A general element of $\H^2(\RE_-) \oplus \H^2(\RE_+)$ can be written as:
$$
\psi= \chi_{\RE_-} \psi_- + \chi_{\RE_+} \psi_+
$$
where $\psi_{\pm} \in \H^2$. Then we have explicitly
\begin{align}
\widehat\delta_{\pm} \psi(x) & = \delta(x) \cdot \psi_{\pm}(x)=\delta(x) \psi_{\pm}(0) \nonumber \\
 \widehat\delta'_{\pm}
\psi(x) & = \delta'(x) \cdot \psi_{\pm}(x)=\delta'(x) \psi_{\pm}(0)-\delta(x) \psi'_{\pm}(0)  \\
\widehat\delta_{\pm} D_x \psi(x) & = \delta(x) \cdot \psi'_{\pm}(x)=\delta(x) \psi'_{\pm}(0)\, . \nonumber
\end{align}

\end{section}

\begin{section}{One-dimensional boundary operators: General results}

Let $\widehat H_0=-D^2_x$ be the one-dimensional free
Schr\"odinger operator with domain on the Sobolev-Hilbert space
$\H^2$. As before, let $\widehat S$ be its symmetric restriction
to the domain $\D(\RE \backslash \{0\})$ and let $\widehat S^*$ be
the adjoint of $\widehat S$. In this section we prove some general
results about one-dimensional Schr\"odinger operators with
arbitrary boundary operators:
\begin{equation}\label{ZB}
\widehat Z:\, \DO(\widehat H) \subseteq \L^2 \longrightarrow \L^2;
\, \, \widehat H=\widehat H_0+\widehat B
\end{equation}
where $\widehat B \in \wh\B$. More precisely, we show that every
$\widehat Z$ of the previous form satisfies $\widehat Z \subseteq
\widehat S^*$, and we determine the maximal domain of $\widehat Z$
in terms of the operator $\widehat B$, explicitly. Moreover, we
show that $\widehat S^*$ can be written in the previous form
(\ref{ZB}), and determine the associated boundary operator
$\widehat B$, explicitly.

Here, and from now on, $\widehat H_0$ is defined in the generalized sense
$$
\widehat H_0:\L^2 \longrightarrow \D'; \quad \widehat H_0=-D^2_x \, .
$$
Moreover, the maximal domain of an arbitrary operator $\widehat Z$ is defined to be
$$
\DO_\max(\widehat Z)\equiv \{\psi \in \L^2: \widehat Z \psi \in \L^2 \}
$$
where the condition $\widehat Z \psi \in \L^2$ means precisely that: i) $\widehat Z \psi$ is well-defined as a
distribution in $\D'$ and ii) there exists $\phi \in \L^2$ such that $\widehat Z \psi=\phi$, weakly.

We start by proving the following general result:

\begin{lemma}\label{LemmaZ}
Let $\widehat Z= \widehat H_0+\widehat B$ where $\widehat B \in
\wh\B$ is an arbitrary boundary operator. Then the maximal domain
of $\widehat Z$ satisfies
$$
\DO_\max(\widehat Z) \subseteq \H^2(\RE_-) \oplus \H^2(\RE_+)\, .
$$
\end{lemma}

\begin{proof}
We have $\psi \in \DO_\max(\widehat Z)$ iff $\psi \in \L^2 \cap \DO(\widehat B)$ and there exists $\phi \in
\L^2$ such that $\widehat Z\psi=\phi$ in the weak sense. This implies that
$$
\langle \widehat Z\psi,t\rangle=\langle \phi,t\rangle \, , \, \forall t \in \D(\RE\backslash \{0\}) \quad
\Longleftrightarrow \quad \left\{
\begin{array}{l}
\left. \left(-D_x^2 \psi \right) \right|_{\RE_-}=\left. \phi \right|_{\RE_-} \\
\left. \left(-D_x^2 \psi \right) \right|_{\RE_+}=\left. \phi \right|_{\RE_+}
\end{array} \right.
$$
where the identities are in the distributional sense and we used the fact that supp $\widehat B \psi\subseteq
\{0\}$. It follows from $\left.(-D_x^2 \psi) \right|_{\RE_\pm}= -D_x^2 \left(\left.\psi
\right|_{\RE_\pm}\right)$ and Grubb's [Theorem 4.20, Remark 4.21 \cite{Grubb}] that $\left. \psi \right|_{\RE_-}
\in \H^2(\RE_-)$, and likewise $\left. \psi \right|_{\RE_+} \in \H^2(\RE_+)$. Since $\psi \in \L^2=\L^2(\RE_-)
\oplus \L^2(\RE_+)$, we conclude that $\psi \in \H^2(\RE_-) \oplus \H^2(\RE_+)$ and so:
$$
\DO_\max(\widehat Z) \subseteq \H^2(\RE_-) \oplus \H^2(\RE_+) \, .
$$
\end{proof}

A natural question is then whether $\widehat Z \subseteq \widehat
S^*$. First we show that $\widehat S^*$ is itself of the form
(\ref{ZB}). Recall that the domain of $\widehat S^*$ is
\cite{Albeverio1,Albeverio2,Albeverio4,Albeverio6}:
$$
\DO(\widehat S^*)= \H^2(\RE_-) \oplus \H^2(\RE_+) \subset \L^2
$$
and that all $\psi \in \H^2(\RE_-) \oplus \H^2(\RE_+)$ can be written in the form:
$$
\psi= \chi_{\RE_-} \psi_- + \chi_{\RE_+} \psi_+ \, ,
$$
where $\psi_-,\psi_+ \in \H^2$. Moreover, $\widehat S^*$ acts as:
$$
\widehat S^* \left[\chi_{\RE_-} \psi_- + \chi_{\RE_+} \psi_+ \right]= - \chi_{\RE_-} \psi''_- - \chi_{\RE_+}
\psi''_+
$$
where, as usual, $\psi''=D_x^2\psi$.

Let us define the operators
$$
\widehat\beta^{(n)}: \A^{n} \longrightarrow \A^{n}; \quad F \longrightarrow
\left[\widehat\delta^{(n)}_+-\widehat\delta^{(n)}_- \right]F
$$
and introduce the notation $\widehat\beta=\widehat\beta^{(0)}$ and $\widehat\beta'=\widehat\beta^{(1)}$. Then

\begin{theorem} \label{TheoS}
The adjoint of $\widehat S$ is given by
\begin{equation}\label{AdjointS}
\widehat S^*=\widehat H_0 + 2 \widehat\beta D_x + \widehat\beta'
\end{equation}
and the domain of $\widehat S^*$ coincides with the maximal domain of the expression on the right hand side,
$$
\DO(\widehat S^*)=\{\psi \in \L^2:\, (\widehat H_0 + 2 \widehat\beta D_x + \widehat\beta')\psi \in \L^2 \} \, .
$$
\end{theorem}

\begin{proof}

Consider the action of $D_x^2$ on $\psi \in \DO(\widehat S^*)$:
\begin{align*}
& D^2_x \left[\chi_{\RE_-} \psi_- + \chi_{\RE_+} \psi_+ \right] \\
=& \chi_{\RE_-} \psi''_- + \chi_{\RE_+} \psi''_++ 2 \delta(x) \left[\psi'_+ - \psi'_- \right] + \delta'(x)
\left[\psi_+ - \psi_- \right]
\end{align*}
and let us re-express the r.h.s in terms of the operators $\widehat\delta_{\pm}$ and $\widehat\delta'_{\pm}$,
$$
-D^2_x \psi =\widehat S^* \psi -2 \left[\widehat\delta_+ -\widehat\delta_-  \right] \psi' -
\left[\widehat\delta'_+-\widehat\delta'_- \right]\psi \, .
$$
Using the operators $\widehat\beta$ and $\widehat\beta'$ we immediately obtain (\ref{AdjointS}).

It remains to prove that the maximal domain of the r.h.s. of
(\ref{AdjointS}) is $\H^2(\RE_-) \oplus \H^2(\RE_+)$. Let us set
$\widehat B=2 \widehat\beta D_x + \widehat\beta'$. Since supp
$\widehat B \psi\subseteq \{0\}$ for all $\psi \in \DO(\widehat
B)$, it follows from Lemma \ref{LemmaZ} that $\DO_\max(\widehat
H_0 + \widehat B) \subseteq \H^2(\RE_-) \oplus \H^2(\RE_+)$.

Moreover, if $\psi \in \H^2(\RE_-) \oplus \H^2(\RE_+)$ then $\psi=\chi_{\RE_-} \psi_- + \chi_{\RE_+} \psi_+$ for
some $\psi_-,\psi_+ \in \H^2$ and
$$
(\widehat H_0 + 2 \widehat\beta D_x + \widehat\beta') \psi=-\chi_{\RE_-} \psi_-'' - \chi_{\RE_+} \psi_+'' \in
\L^2 \, .
$$
Hence, $\DO_\max (\widehat H_0 + 2 \widehat\beta D_x + \widehat\beta')=\H^2(\RE_-) \oplus
\H^2(\RE_+)=\DO(\widehat S^*)$.
\end{proof}

We are now able to prove the main result of this section:

\begin{theorem}\label{TheoA}
Let $\widehat B \in \wh\B$ be an arbitrary boundary operator and
let
$$
\widehat Z: \DO_\max(\widehat Z) \subset \L^2\longrightarrow \L^2,\quad \widehat Z= \widehat H_0 +\widehat B \,
.
$$
Then $\widehat Z\subseteq \widehat S^*$, and
$$
\DO_\max(\widehat Z)= \mbox{Ker } \widehat F \, \, \cap \, \,
\left( \H^2(\RE_-) \oplus \H^2(\RE_+)\right)
$$
where
\begin{equation}\label{OperatorF}
\widehat F = -2 \widehat\beta D_x - \widehat\beta' +\widehat B \, .
\end{equation}
\end{theorem}

\begin{proof}

Since supp $\widehat B \psi \subseteq \{0\}$ for all $\psi \in
\DO(\widehat B)$, we have from Lemma \ref{LemmaZ}
$$
\DO_\max(\widehat Z) \subseteq \H^2(\RE_-) \oplus \H^2(\RE_+).
$$
It follows from Theorem \ref{TheoS} (and the definition of
$\widehat Z$) that on $\DO_\max(\widehat Z)$
$$
\widehat Z= \widehat S^* -2 \widehat\beta D_x - \widehat\beta' +
\widehat B= \widehat S^* + \widehat F
$$
where $\widehat F$ is given by (\ref{OperatorF}). Since supp $\widehat F \psi \subseteq \{0\}$, the term
$\widehat F \psi$ is a linear combination of a Dirac delta and its derivatives and so
$$
\psi \in \DO_\max(\wh Z) \Longrightarrow (\widehat S^* + \widehat
F) \psi \in \L^2 \Longrightarrow \widehat F \psi =0 \, .
$$
Hence, $\DO_\max(\widehat Z) \subseteq $ Ker $\widehat F$.
Conversely
$$
\psi \in \mbox{Ker } \widehat F \, \, \cap \, \, \left(\H^2(\RE_-)
\oplus \H^2(\RE_+) \right) \Longrightarrow \widehat Z \psi =
\widehat S^* \psi \in \L^2 \, .
$$
We conclude that $\DO_\max(\widehat Z)=$ Ker $\widehat F  \, \,
\cap \, \, \left(\H^2(\RE_-) \oplus \H^2(\RE_+) \right)$ and
$\widehat Z \subseteq \widehat S^*$.

\end{proof}

\end{section}

\begin{section}{One-dimensional Schr\"odinger operators with point interactions}

Let now $\widehat Z$ denote an arbitrary restriction of $\widehat
S^*$ to a domain characterized by two local boundary conditions at
$x=0$:
\begin{eqnarray} \label{Z}
& \widehat Z: \, \DO(\widehat Z) \subseteq \L^2 \longrightarrow \L^2; \, \widehat Z \psi = \widehat S^* \psi &
\\
& \DO(\widehat Z)=\left\{\psi \in \H^2(\RE_-) \oplus \H^2(\RE_+) :
\, f_i(\psi_\pm(0),\psi_\pm'(0))=0 \, , i=1,2\right\} & \nonumber
\end{eqnarray}
where $\psi_\pm(0)= (\psi_+(0),\psi_-(0))$, $\psi_\pm'(0)= (\psi_+'(0),\psi_-'(0))$ and $f_i: \CO^4
\longrightarrow \CO$, $i=1,2$, are linear functions.

In this section we show that every operator $\widehat Z$ can be written in the form
\begin{equation} \label{Z2}
\widehat Z=\widehat H_0 +\widehat B
\end{equation}
$$
\DO(\widehat Z)=\DO_\max(\widehat H_0+\widehat B)
$$
where $\widehat B \in \wh\P$ is a {\it boundary pseudo potential
operator} of the form (\ref{BPP1}).

The one-dimensional Schr\"odinger operators with point
interactions $\widehat L$ are all of the form (\ref{Z}). For each
$\widehat L$, we will calculate a boundary pseudo potential
representation (\ref{Z2}) explicitly (Corollaries \ref{CorrS},
\ref{CorrDirichlet} and \ref{CorrI}).

We remark that for each $\widehat Z$ (\ref{Z}) there are, in
general, several different operators $\widehat B \in \wh\P$ such
that (\ref{Z2}) is valid. This will be shown explicitly for the
operator $\widehat L$ with Dirichlet boundary conditions at $x=0$
(Corollaries \ref{CorrS} and \ref{CorrDirichlet}). In the next
section, the non-uniqueness of $\wh B$ will be studied in more
detail.

A natural question is whether the operators $\widehat Z$ (or at
least the operators $\widehat L$) admit a (simpler) boundary
potential representation (and not only a boundary pseudo potential
representation), i.e. a representation of the form (\ref{Z2}) with
$\widehat B \in \widehat\A^1$. While this is true for a large
class of operators $\widehat L$ (see Theorems \ref{TheoIn} and
\ref{TheoSe}, in the next section) the following Theorem shows
that it is not true for all $\widehat L$, not even if we only
require $\wh B$ to be of the form (\ref{BPP1}) with $\widehat
B_3=0$.

\begin{theorem}

The set of operators $\widehat Z=\widehat H_0 + \widehat B$
(\ref{Z2}) where $\widehat B$ is of the form (\ref{BPP1}) with
$\widehat B_3=0$, does not contain all s.a. extensions of
$\widehat S$.

\end{theorem}

\begin{proof}
If $\wh B_3=0$ then $\wh B \psi = \wh B_1 \psi + \wh B_2 \psi'$
and $\widehat F= -2 \wh\beta D_x-\wh\beta'+\wh B$, given by
(\ref{OperatorF}), is also of the form
$$
\wh F \psi= \wh F_1 \psi + \wh F_2 \psi'
$$
for some $\wh F_1 \in \wh\A^1$ and $\wh F_2 \in \wh\A^0$.

On the other hand, if $\wh Z \subseteq \wh S^*$ then supp $\wh B
\psi \subseteq \{0\}$ for all $\psi \in \DO(\wh B)$. The same is
then true for $\wh F_1, \wh F_2$ and so
$$
\wh F_i= a_i \wh\delta_- + b_i
\wh\delta_+ +c_i \wh\delta'_- +d_i \wh\delta'_+
$$
for some $a_i,b_i,c_i,d_i \in \CO$, $i=1,2$ and $c_2=d_2=0$.

We have from Theorem \ref{TheoA} that $\DO(\wh Z)=$ Ker $\wh F \,
\cap \, \left(\H^2(\RE_-) \oplus \H^2(\RE_+)\right)$ and so $\psi
\in \DO(\wh Z)$ satisfies:
$$
\wh F \psi =0 \Longleftrightarrow \delta(x) f_1( \psi_\pm (x),\psi_\pm'(x)) + \delta'(x) f_2( \psi_\pm (x)) =0
$$
for some linear functions $f_1: \CO^4 \longrightarrow \CO$ and
$f_2: \CO^2 \longrightarrow \CO$. This is equivalent to:
$$
\delta(x) \left[ f_1( \psi_\pm (x),\psi_\pm'(x))- f_2( \psi'_\pm
(x))\right]+D_x \left[ \delta(x) f_2( \psi_\pm (x)) \right]=0 \, .
$$
The two terms on the l.h.s. are linear independent and so:
$$
\left\{\begin{array}{l} f_1( \psi_\pm (0),\psi_\pm'(0))- f_2( \psi'_\pm (0))=0 \\
f_2( \psi_\pm (0))=0
\end{array}
\right.
$$
These conditions are unable to implement any two boundary conditions that involve two linear independent
combinations of $\psi'_-(0)$ and $\psi_+'(0)$. This is the case, for instance, of the conditions:
$$
\psi'_-(0)=0 \qquad \mbox{and} \qquad \psi'_+(0)=0
$$
which correspond to the operator $\widehat L$ with Neumann
boundary conditions at both sides of the boundary at $x=0$. Hence,
it is not possible to construct a boundary potential formulation
of all operators $\wh L$.

\end{proof}

We remark that the more general possibility $\wh B_2 \in \wh\A^1
\backslash \wh\A^0$ cannot be considered, because then
$\H^2(\RE_-) \oplus \H^2(\RE_+) \not\subseteq \DO (\wh B_2 D_x)$.
Instead, $\wh B_2 \in \wh\A^0$, which implies $\wh F_2 \in
\wh\A^0$ and so $c_2=d_2=0$.

Our approach to obtain the representation (\ref{Z2}) is then based on the following general result, which is a
Corollary of Theorem \ref{TheoA}:

\begin{corollary} \label{CorrZ}
Let $\widehat Z \subseteq \widehat S^*$. If exists
$$
\widehat F \, : \DO(\widehat F) \subseteq \D' \longrightarrow \D'
$$
such that:

(i) supp $\widehat F \psi \subseteq \{0\}$ for all $\psi \in \DO(\widehat F)$;

(ii)  Ker $(\widehat F)  \cap \left( \H^2(\RE_-) \oplus \H^2(\RE_+) \right)=\DO(\widehat Z)$,\\
then $\widehat Z$ admits the representation
\begin{equation}
\widehat Z:\DO(\widehat Z)\subset \L^2 \longrightarrow \L^2; \quad \widehat Z=\widehat H_0 +\widehat B
\label{OperatorZ}
\end{equation}
$$
\DO(\widehat Z)=\DO_\max(\widehat H_0 +\widehat B)
$$
where
$$
\widehat B= 2 \widehat\beta D_x + \widehat\beta' + \widehat F \, .
$$

\end{corollary}

\begin{proof}
It follows from the definition of $\wh B$ that if supp $\wh F \psi
\subseteq \{0\}$ for all $\psi \in \DO(\wh F)$, then also supp
$\wh B \psi \subseteq \{0\}$ for all $\psi \in \DO(\wh B)$.

Consider the operator $\wh H_0 + \wh B$. From Theorem \ref{TheoA}
$$
\DO_{max} (\wh H_0 + \wh B)= \mbox{Ker} \, \, (\wh B -2 \wh\beta
D_x -\wh\beta') \cap \left(\H^2(\RE_-) \oplus \H^2(\RE_+) \right)
$$
$$
\ \ \ \ \ \ \ \ \ \ \ \ = \mbox{Ker} \, \, (\wh F) \cap
\left(\H^2(\RE_-) \oplus \H^2(\RE_+)\right) = \DO(\wh Z)
$$
which proves the second formula in (\ref{OperatorZ}).

Since, by assumption, $\wh Z \subseteq \wh S^*$, and from Theorem \ref{TheoA} also $\wh H_0 +\wh B \subseteq \wh
S^*$, we conclude that $\wh Z= \wh H_0+\wh B$.

\end{proof}

The main point in determining a boundary operator representation for all operators $\widehat Z$ (\ref{Z}) is
then to determine, for each $\wh Z$, a suitable operator $\widehat F$. In general, there are many possible
choices of the operator $\widehat F$ (and consequently of $\widehat B$). Since every $\widehat Z$ is
characterized by two boundary conditions at $x=0$, one natural possibility is the following:

\begin{theorem}\label{FF}
Let $\wh Z$ be the operator (\ref{Z}), and let
$$
\widehat F=\widehat F_1 + D_x \widehat F_2
$$
where
$$
\widehat F_i : \H^2(\RE_-) \oplus \H^2(\RE_+) \subset \A^1
\longrightarrow \D' \quad , \quad i=1,2
$$
are the linear operators acting as ($\wh\delta_\pm
=(\wh\delta_+,\wh\delta_-)$):
$$
\widehat F_i \psi = f_i(\wh\delta_\pm,\wh\delta_\pm D_x) \psi
$$
and $f_i: \CO^4 \longrightarrow \CO$, $i=1,2$ are the linear
functions that impose the boundary conditions of $\DO(\widehat Z)$
(\ref{Z}).

 Then $\wh F$ satisfies the conditions (i) and (ii) of
Corollary \ref{CorrZ}:

(i) supp $\widehat F \psi \subseteq \{0\}$ for all $\psi \in
\DO(\wh F)$;

(ii) Ker $\widehat F \, \, \cap \, \, \left(\H^2(\RE_-) \oplus
\H^2(\RE_+) \right)=\DO(\widehat Z)$.

\end{theorem}

\begin{proof}

We start by noticing that for all $\psi \in \DO(\wh F)$
$$
\wh F_i \psi = \delta (x) f_i(\psi_\pm(0),\psi_\pm'(0))\, , \quad
i=1,2 \, .
$$
Hence, supp $(\widehat F_i\psi) \subseteq \{0\}$ and the condition (i) is satisfied. Moreover,
\begin{align*}
\psi \in \text{Ker} \, \widehat F &\Longleftrightarrow \delta(x) f_1(\psi_\pm(0),\psi_\pm'(0))+
\delta'(x)f_2(\psi_\pm(0),\psi_\pm'(0))=0 \\
&\Longleftrightarrow  f_1(\psi_\pm(0),\psi_\pm'(0))=f_2(\psi_\pm(0),\psi_\pm'(0))=0 \, .
\end{align*}
Hence, Ker $(\widehat F) \, \, \cap \, \, \left(\H^2(\RE_-) \oplus
\H^2(\RE_+)\right) =\DO(\widehat Z)$ and so $\widehat F$ also
satisfies the condition (ii).

\end{proof}

We conclude from Corollary \ref{CorrZ} and the previous Theorem that each operator $\widehat Z$ (\ref{Z}) admits
a boundary pseudo potential representation of the form (\ref{Z2}), with
$$
\widehat B  = \wh\beta'  + 2 \wh\beta D_x + f_1(\wh\delta_\pm ,\wh\delta_\pm D_x) + D_x f_2(\wh\delta_\pm
,\wh\delta_\pm D_x) \, .
$$

\begin{subsection}*{Boundary pseudo potential representation of the operators $\widehat L$}

The previous results are now used to determine a boundary pseudo potential representation for each $\wh L$,
explicitly. It is also shown, using a particular example, that each $\widehat L$ may admit more than one
representation of the form (\ref{Z2}) and, in particular, it may admit a simpler {\it boundary potential}
representation (this question will be studied in detail in the next section).

All operators $\widehat L$ are s.a. restrictions of $\widehat S^*$
of the form of $\widehat Z$ (\ref{Z}). The two boundary conditions
that characterize $\DO(\widehat L)$ can be \underline{separating},
in which case they can be written as \cite{Kurasov1,Seba1}:
\begin{equation} \label{SBC}
\left\{ \begin{array}{l} a_- \psi'(0^-)=b_- \psi(0^-)\\
a_+\psi'(0^+)=b_+ \psi(0^+) \end{array} \right. \quad , \quad
(a_{\pm},b_{\pm})\in \RE^2 \backslash \{(0,0)\}
\end{equation}
and lead to {\it confining} Schr\"odinger operators of the form $\widehat L=\widehat L_- \oplus \widehat
L_+$:
$$
\widehat L_- \oplus \widehat L_+ : \DO(\widehat L_-) \oplus \DO
(\widehat L_+) \longrightarrow \L^2; \quad \psi \longrightarrow
(\widehat L_- \oplus \widehat L_+) \psi= \widehat S^* \psi
$$
where
$$
\DO(\widehat L_{\pm}) =\{ \psi_\pm=  \chi_{\RE_{\pm}}\psi:\, \psi
\in \H^2 \, \wedge a_{\pm} \psi'(0) = b_{\pm} \psi(0) \} \, .
$$
Notice that $\widehat L_{\pm}$ are s.a. extensions of the restrictions of $\widehat H_0$ to $\D(\RE_{\pm})$,
\cite{Dias3}. The operators $\widehat L$ of the form $\widehat L_- \oplus \widehat L_+$ commute with the
projection operators $\widehat\chi_{\RE_{\pm}}=\chi_{\RE_{\pm}}\cdot$ and provide a "global" description of
quantum systems confined to either of the domains $\RE_-$ or $\RE_+$ \cite{Dias3}.

The other possibility is that the s.a. boundary conditions are \underline{interacting} ($a,c \in \RE$, $b\in
\CO$: $(1+\ol b)(1-b) -ac \not=0$), \cite{Albeverio4}:
\begin{equation} \label{IBC}
\left\{ \begin{array}{l}
\psi(0^+)-\psi(0^-)  =  a \left(\psi'(0^+)+\psi'(0^-) \right) + \overline{b} \left(\psi(0^+)+\psi(0^-)\right) \\
\psi'(0^+)-\psi'(0^-)  =  c \left(\psi(0^+)+\psi(0^-)\right) - b \left(\psi'(0^+)+\psi'(0^-)\right)
\end{array}  \right.
\end{equation}
in which case they relate the values of the wave function at the two sides of the boundary. The associated
operator $\widehat L$ cannot be written in the form $\widehat L_- \oplus \widehat L_+$. This kind of operators
describe quantum systems formed by two sub-systems which are not isolated from each other (as in the case of
separating boundary conditions) but instead display some sort of interaction at their common boundary.

In this section we construct a boundary pseudo potential
formulation of both the separating and interacting operators
$\widehat L$. Using the prescription of Theorem \ref{FF}, we get

\begin{corollary} \label{CorrS}
The {\it separating} Schr\"odinger operators can be written in the
form:
$$
\widehat L^S:\DO(\widehat L^S) \subseteq \L^2 \longrightarrow
\L^2, \quad \widehat L^S=\widehat H_0+\widehat B^S
$$
$$
\DO(\widehat L^S)=\DO_\max(\widehat H_0+\widehat B^S)
$$
where
$$
\widehat B^S:\H^2(\RE_-) \oplus \H^2(\RE_+) \longrightarrow \D',
\quad \widehat B^S=2\widehat\beta D_x + \widehat\beta' + \widehat
F^S
$$
and
$$
\widehat F^S:\H^2(\RE_-) \oplus \H^2(\RE_+) \longrightarrow \D',
\quad \widehat F^S=\widehat F_+^S+D_x \widehat F_-^S
$$
$$
\widehat F_{\pm}^S:\H^2(\RE_-) \oplus \H^2(\RE_+) \longrightarrow
\D', \quad \widehat F^S_{\pm}= a_{\pm} \widehat\delta_{\pm} D_x
-b_{\pm} \widehat\delta_{\pm}
$$
\end{corollary}

\begin{proof}
The proof follows from the Corollary \ref{CorrZ} since supp
$\widehat F_{\pm}^S \psi \subseteq \{0\}$ and
\begin{align*}
\widehat F^S \psi =0  \Longleftrightarrow \widehat F_{\pm}^S \psi =0 & \Longleftrightarrow a_{\pm} \delta (x)
\psi'_{\pm}(x) -b_{\pm} \delta (x) \psi_{\pm} (x)=0 \\
& \Longleftrightarrow a_{\pm} \psi'_{\pm}(0) -b_{\pm} \psi_{\pm}
(0)=0
\end{align*}
which shows that Ker $\widehat F^S=\DO(\widehat L^S)$.
\end{proof}

The previous result is valid for general separating boundary conditions. For particular cases, the expression of
$\widehat L^S$ simplifies considerably. For instance, for Dirichlet boundary conditions (i.e. $a_{\pm}=0$,
$b_{\pm}=1$) the operator $\widehat L^S$ becomes
$$
\widehat L^D=-D_x^2+2\widehat\beta D_x + \widehat\beta' +\widehat F^D_1
$$
where
$$
\widehat F^D_1=-\widehat\delta_+ -D_x \widehat\delta_- \, .
$$

As we have already pointed out, for each $\widehat L$, there are many possible choices of the operator $\widehat
F$. To illustrate this let us introduce the operator $\widehat\alpha$ and its "derivatives"
$\widehat\alpha^{(n)}$:
$$
\widehat\alpha^{(n)} : \A^{n} \longrightarrow \A^{n}, \quad
\widehat\alpha^{(n)}=\widehat\delta^{(n)}_+ +
\widehat\delta^{(n)}_- \, .
$$
We then have, for instance

\begin{corollary}\label{CorrDirichlet}
The Schr\"odinger operator satisfying Dirichlet boundary
conditions at both sides of the boundary at $x=0$, i.e.
$\psi(0^{\pm})=0$, can also be written as
$$
\widehat L^D: \DO(\widehat L^D) \subseteq \L^2 \longrightarrow \L^2, \quad \widehat L^D=\widehat H_0+
\widehat\alpha -\widehat\beta'
$$
$$
\DO(\widehat L^D)=\DO_\max(\widehat H_0+ \widehat\alpha -\widehat\beta')
$$
which yields a boundary potential representation of $\wh L^D$.
\end{corollary}

\begin{proof}

Let us define
$$
\widehat F^D_2: \H^2(\RE_-) \oplus \H^2(\RE_+) \longrightarrow \D',\quad \widehat F^D_2=\widehat\alpha -2D_x
\widehat\beta \, .
$$
Then supp $(\widehat F^D_2 \psi) \subseteq \{0\}$, $\forall \psi \in \DO(\wh F^D_2)$ and
$$
\widehat F^D_2 \psi=0 \Longleftrightarrow \left\{
\begin{array}{l}
\delta(x) (\psi_+(0)+\psi_-(0))=0\\
\delta'(x) (\psi_+(0)-\psi_-(0))=0 \end{array} \right.
\Longleftrightarrow \psi_+(0)=\psi_-(0)=0 \, .
$$
Hence, Ker $\widehat F^D_2=\DO(\widehat L^D)$ and so $\widehat
F^D_2$ satisfies the two conditions of Corollary \ref{CorrZ}. The
proof is concluded by
$$
\widehat H_0 +2 \widehat\beta D_x + \widehat\beta' + \widehat F^D_2= \widehat H_0+ \widehat\alpha
-\widehat\beta' \, .
$$
\end{proof}

Finally, let $\widehat L^I$ be the Schr\"odinger operator
satisfying the {\it interacting} boundary conditions (\ref{IBC})
at $x=0$. A boundary potential representation of $\widehat L^I$
can also be determined using the method of Theorem \ref{FF}:

\begin{corollary}\label{CorrI}
The operators $\widehat L^I$ admit the representation
$$
\widehat L^I:\DO(\widehat L^I) \subseteq \L^2 \longrightarrow
\L^2, \quad \widehat L^I=\widehat H_0+\widehat B^I
$$
$$
\DO(\widehat L^I)=\DO_\max(\widehat H_0+\widehat B^I)
$$
where
$$
\widehat B^I:\H^2(\RE_-) \oplus \H^2(\RE_+) \longrightarrow \D', \quad \widehat B^I=c \widehat\alpha -b
\widehat\alpha D_x +a D_x \widehat\alpha D_x + \overline{b} D_x \widehat\alpha \, .
$$
\end{corollary}

\begin{proof}
Setting $\widehat F= \widehat F_1+ D_x  \widehat F_2$ with
$$
\widehat F_1: \H^2(\RE_-) \oplus \H^2(\RE_+) \longrightarrow \D', \quad \widehat F_1= c \widehat\alpha -b
\widehat\alpha D_x - \widehat\beta D_x
$$
$$
\widehat F_2: \H^2(\RE_-) \oplus \H^2(\RE_+) \longrightarrow \D', \quad \widehat F_2= a \widehat\alpha D_x
+\overline{b} \widehat\alpha - \widehat\beta
$$
we have supp $\widehat F \psi \subseteq \{0\}$ for all $\psi \in \DO(\widehat F)$, and
$$
\DO(\widehat L^I) = {\rm Ker} (\widehat F_1) \cap {\rm Ker} (\widehat F_2) = {\rm Ker} (\widehat F) \, .
$$
It follows that
$$
\widehat L^I= \widehat S^* + \widehat F = -D_x^2 + 2 \widehat\beta D_x +\widehat\beta'+ c \widehat\alpha -b
\widehat\alpha D_x - \widehat\beta D_x + D_x(a \widehat\alpha D_x +\overline{b} \widehat\alpha - \widehat\beta)
$$
$$
=\widehat H_0 +\widehat B^I
$$
which concludes the proof.
\end{proof}

\end{subsection}

\end{section}

\begin{section}{The $a\delta+b\delta'$ potential}

In this section we address a problem which, in some sense, is the inverse of the one studied in the previous
section. We are given a singular boundary potential $B$ and the aim is to determine the explicit form of the
operators $\widehat H_0+ \widehat B$, where $\widehat B$ is a boundary operator associated with $B$.

The crucial point here is the definition of the association between $B$ and $\widehat B$. Recall that the
simplest definition $\widehat B=B \cdot$ (where $\cdot$ is the standard product of a distribution by a test
function; or some obvious extension of it) yields boundary operators with very restricted domains.

Other ways of implementing the boundary potential $B$ yield a
richer structure. As we have already mentioned in the
introduction, one possible interpretation of $\widehat H=\widehat
H_0+ \widehat B$ is that it stands for the norm resolvent limit of
a sequence of operators $\widehat H_n=\widehat H_0+\widehat B_n$,
where $\widehat B_n=B_n \cdot$ and $B_n$ is a sequence of regular
potentials such that $B_n \longrightarrow B$ in $\D'$. The case
$B=a\delta(x)+b\delta'(x)$, with $a,b \in \RE$, has been
extensively studied in the literature (see
\cite{Golovaty3,Zolotaryuk1} and the references therein). It turns
out that for $B_n \longrightarrow B=a \delta(x)$ in $\D'$, the
norm resolvent limit of $\widehat H_n$ is (for a large class of
regular potentials $B_n$)
\begin{equation}\label{K}
\widehat H_{a\delta}:\DO(\widehat H_{a\delta}) \subset \L^2
\longrightarrow \L^2, \quad \psi \longrightarrow \widehat
H_{a\delta} \psi =\widehat S^* \psi
\end{equation}
$$
\DO(\widehat H_{a\delta})=\left\{  \psi \in \H^2(\RE_-) \oplus \H^2(\RE_+): \,  \left\{
\begin{array}{l} \psi(0^+)=\psi(0^-) \\
 \psi'(0^+)-\psi'(0^-) = a \psi(0) \end{array} \right. \right\}
$$
and is independent of the particular sequence $B_n$ such that $B_n \longrightarrow a\delta(x)$.

The case is different if $B=b \delta'(x)$. Golovaty, Man'ko and Hryniv \cite{Golovaty1,Golovaty2,Golovaty3} and
Zolotaryuk \cite{Zolotaryuk1} determined families of sequences $B_n \longrightarrow B$, displaying a single
distributional limit $B=b\delta'(x)$ but yielding, in the norm resolvent sense, the family of limit operators:
\begin{equation}\label{K'}
\widehat H_{b\delta',\theta}:\DO(\widehat H_{b\delta',\theta})
\subset \L^2 \longrightarrow \L^2, \quad \psi \longrightarrow
\widehat H_{b\delta',\theta} \psi =\widehat S^* \psi
\end{equation}
$$
\DO(\widehat H_{b\delta',\theta})=\{\psi \in \H^2(\RE_-) \oplus \H^2(\RE_+): \,  \psi(0^+)=\theta\psi(0^-) \,
\wedge \, \theta \psi'(0^+)=\psi'(0^-) \}
$$
where the parameter $\theta$ depends on the shape of the potentials $B_n$.

In this section we will study the Schr\"odinger operators $\wh
H=\wh H_0 + \wh B$, where $\wh B $ is a {\it boundary potential
operator} in $\wh\A^1$:
$$
\wh B = \psi* B_1 + B_2 * \psi
$$
with $B_1, B_2 \in \A^1$. The operator $\wh B$ provides a natural
operator representation of the boundary potential $B=B_1+B_2=a
\delta(x) + b\delta'(x)$. Notice that $\wh B$ is an extension of
$B \cdot$ to $\A^1 \supset \H^2(\RE_-) \oplus \H^2(\RE_+)$.

In subsection 6.1, we obtain in Theorem \ref{CorrH} the explicit
form of the operators $\wh H=\wh H_0+\wh B$. Then in Theorems
\ref{TheoIn}, \ref{TheoSe} and Corollary \ref{CorrInt}, we
determine which operators $\wh H$ are s.a. and conversely, which
s.a. extensions of $\wh S$ admit a boundary potential
representation of the form $\wh H$. We will see that the two sets
of operators (the ones of the form $\wh H=\wh H_0+\wh B$, and the
Schr\"odinger operators $\wh L$) have a large intersection, but do
not coincide, nor one contains the other. In Theorems
\ref{TheoIn}, \ref{TheoSe} and Corollary \ref{CorrInt}, we also
determine, for the s.a. case, the entire set of boundary potential
operators $\wh B \in \wh\A^1$ that yield a single operator $\wh
L$. Finally, in Theorem \ref{SQF}, we calculate the sesquilinear
form associated with $\wh H$.

In subsection 6.2, we consider the particular cases $B=a\delta(x)$
and $B=b\delta'(x)$, and compare the results of the boundary
operator formulations with the results of the norm resolvent
approach (Corollaries \ref{CorrD}, \ref{CorrD'1} and
\ref{CorrD'2}).

\begin{subsection}{Schr\"odinger operators with boundary potentials}

Let $\wh B \in \wh\A^1$ be of the form (\ref{B0}) with $B_i=c_i
\delta(x) + b_i \delta'(x)$, $c_i,b_i \in \CO$, $i=1,2$. Then:
\begin{equation} \label{B1}
\wh B= c_1 \wh\delta_- + c_2 \wh\delta_+ +b_1 \wh\delta_-' + b_2 \wh\delta_+'
\end{equation}
and $\wh B \leftrightarrow B=c \delta(x) + b \delta'(x)$, where
$c=c_1+c_2$ and $b=b_1+b_2$. The following Theorem completely
characterizes the action and the domain of the operators $\wh
H=\wh H_0 +\wh B$.

\begin{theorem} \label{CorrH}
Let $\wh H= \wh H_0 + \wh B$ where $\wh B$ is given by (\ref{B1}). Then

(i) $\wh H \subseteq \wh S^*$;\\

(ii) $\psi \in \DO_{max}(\wh H)$ iff $\psi \in \H^2(\RE_-) \oplus
\H^2(\RE_+)$ and
\begin{equation} \label{BC1}
\left[ \begin{array}{cccc}
-c_1 & -c_2 & (b_1-1) & (b_2+1)  \\
(b_1+1) & (b_2-1) & 0 & 0
\end{array} \right] \left[ \begin{array}{l}
\psi_-(0) \\ \psi_+(0) \\ \psi_-'(0) \\ \psi_+'(0)
\end{array}
\right]=0
\end{equation}
where we wrote, as usual, $\psi=\chi_{\RE_-} \psi_- + \chi_{\RE_+}
\psi_+$, $\psi_\pm \in \H^2$.\\

\end{theorem}

\begin{proof}
Since supp $\wh B \psi \subseteq \{0\}$ for all $\psi \in \DO(\wh
B) =\A^1$, it follows from Theorem \ref{TheoA} that $\wh H
\subseteq \wh S^*$. Moreover, also from Theorem \ref{TheoA},
$$
\DO_{max}(\wh H)= \,\, \mbox{Ker} \, \,  \wh F \, \cap \,
\left(\H^2(\RE_-) \oplus \H^2(\RE_+) \right)
$$
where $\wh F$ is given by
(\ref{OperatorF}):
\begin{eqnarray}
\wh F \psi & = & -2 \wh\beta D_x \psi- \wh\beta' \psi+ \wh B \psi\nonumber \\
&=& \delta (x) \left[ c_1 \psi_- +c_2 \psi_+ + 2 \psi_-' - 2 \psi_+' \right] + \delta' (x) \left[ (b_1+1) \psi_-
+(b_2-1) \psi_+ \right] \, .\nonumber
\end{eqnarray}
Since $\delta'(x) \psi_\pm = D_x (\delta(x) \psi_\pm) -\delta(x) \psi'_\pm$, we easily get:
\begin{eqnarray}
\wh F \psi &=& \delta (x) \left[ c_1 \psi_- +c_2 \psi_+ -(b_1-1) \psi_-' - (b_2+1) \psi_+' \right] \nonumber \\
& & + D_x \left[\delta (x) \left( (b_1+1) \psi_- +(b_2-1) \psi_+ \right) \right] \nonumber
\end{eqnarray}
and so
$$
\wh F \psi =0 \Longleftrightarrow \left\{
\begin{array}{l}
c_1 \psi_-(0) + c_2\psi_+(0) -(b_1-1) \psi_-'(0) -(b_2+1) \psi_+'(0)=0 \\
(b_1+1) \psi_-(0) +(b_2-1) \psi_+(0)=0
\end{array} \right.
$$
which is equivalent to the condition (\ref{BC1}).

\end{proof}

A natural question is then which operators $\wh H= \wh H_0 + \wh
B$ are s.a., and conversely, which s.a. restrictions of $\wh S^*$
admit a boundary potential representation $\wh H_0+ \wh B$ with
$\wh B$ of the form (\ref{B1}). We start by recalling that $\wh Z
\subseteq \wh S^*$ is s.a. iff $\DO(\wh Z) \subset \H^2(\RE_-)
\oplus \H^2(\RE_+)$ is characterized by two {\it separating}
boundary conditions (\ref{SBC}):
\begin{equation} \label{BC2}
\left[ \begin{array}{cccc}
b_- & 0 & -a_- & 0  \\
0 & b_+ & 0 & -a_+
\end{array} \right] \left[ \begin{array}{l}
\psi_-(0) \\ \psi_+(0) \\ \psi_-'(0) \\ \psi_+'(0)
\end{array}
\right]=0 \quad ; \quad (a_\pm,b_\pm) \in \RE^2 \backslash
\{(0,0)\}
\end{equation}
in which case the operator is denoted by $\wh L^S$ or, more
explicitly, by $\wh L^S_{(a_-,a_+,b_-,b_+)}$. Alternatively, $\wh
Z$ might satisfy two {\it interacting} boundary conditions of the
form (\ref{IBC}):
\begin{equation} \label{BC3}
\left[ \begin{array}{cccc}
-c & -c & (b-1) & (b+1)  \\
(\ol b+1) & (\ol b-1) & a & a
\end{array} \right] \left[ \begin{array}{l}
\psi_-(0) \\ \psi_+(0) \\ \psi_-'(0) \\ \psi_+'(0)
\end{array}
\right]=0
\end{equation}
$$
a,c \in \RE \quad , \quad b \in \CO  \, : \, (\ol b+1)(1-b)-ac
\not=0
$$
in which case the operator $\wh Z$ is denoted by $\wh L^I$ or,
more explicitly, by $\wh L^I_{(a,b,c)}$.

The two following theorems study the relation between the
operators $\wh H = \wh H_0 + \wh B$ and the s.a. operators $\wh
L^I$ and $\wh L^S$.

\begin{theorem}\label{TheoIn}
(1) Let $\wh H=\wh H_0 + \wh B$ where $\wh B$ is given by
(\ref{B1}). For arbitrary $(c_1,c_2,b_1,b_2) \in \CO^4$, the
operator $\wh H$ is an {\it interacting} s.a. Schr\"odinger
operator $\wh L^I_{(a,b,c)} \subseteq \wh S^*$
iff either (1a) or (1b) holds true:\\
\\
(1a) $b_1= \ol b_2$, $b_1 \not=\pm 1$ and Im $\left[ c_1(\ol b_1-1) -c_2(b_1+1) \right] =0$. In this case, $\wh
H=\wh L^I_{(a,b,c)}$ with:
\begin{equation}\label{ABC1}
a=0 \quad , \quad b= \frac{b_1+\ol b_1}{\ol b_1-b_1 +2} \quad , \quad c= \frac{2c_1(\ol b_1-1)
-2c_2(b_1+1)}{(\ol b_1-b_1)^2 -4} \, .
\end{equation}

Alternatively:\\
\\
(1b) $b_1+b_2=0$, $b_1 \not=\pm 1$, and Im $\left[ \frac{c_1+c_2}{2(1-b_1)} \right] =0$. In this case $\wh H=\wh
L^I_{(a,b,c)}$ with:
\begin{equation}\label{ABC2}
a=0 \qquad , \qquad b=0 \qquad , \qquad c= \frac{c_1 +c_2}{2(1-b_1)} \, .
\end{equation}
\\
(2) Conversely, $\wh L^I_{(a,b,c)}$ admits a boundary potential representation of the form $\wh H_0 + \wh B$,
with $\wh B$ given
by (\ref{B1}), iff one of the following holds true:\\
\\
(2a) $a=0$, $b+\ol b \not=0$ and $b \not=\pm 1$. Then $\wh B$ has
parameters satisfying the conditions:
\begin{equation} \label{abc2}
b_1= \frac{2b \ol b + b - \ol b}{b + \ol b} \quad , \quad  b_2=\ol b_1
\end{equation}
and
\begin{equation}\label{abc4}
(c_1,c_2)=(k_1/X_1,k_2/X_2) \quad , \quad k_1,k_2 \in \CO : \, k_1+k_2 =c
\end{equation}
where
\begin{equation}\label{abc5}
X_1=-\frac{b + \ol b}{4} + \frac{b + \ol b}{4 \ol b} \quad , \quad X_2 = \frac{b + \ol b}{4} + \frac{b + \ol
b}{4 \ol b} \, .
\end{equation}
\\
(2b) $a,b=0$. In this case the parameters $(c_1,c_2,b_1,b_2) \in \CO^4$ satisfy the conditions:
\begin{equation*}
b_1 \in \CO \backslash \{-1,1\} \quad , \quad  b_2=-b_1 \quad , \quad  c_1+c_2=2 c (1-b_1) \, .
\end{equation*}

\end{theorem}

\begin{proof}
The operators $\wh L^I$ and $\wh H=\wh H_0+\wh B$ are both
restrictions of $\wh S^*$. Hence, $\wh L^I=\wh H$ iff their
domains are the same. This is true iff the boundary conditions
(\ref{BC1}) and (\ref{BC3}) are equivalent.

Since the two equations in (\ref{BC3}) are linearly independent,
the two sets of boundary conditions (\ref{BC1}) and (\ref{BC3})
are equivalent iff exists $\la_1, \la_2, \mu_1,\mu_2 \in \CO$ such
that
\begin{equation} \label{S1}
\left\{ \begin{array}{l} \la_1 (-c_1,-c_2,b_1-1,b_2+1)+ \mu_1 (b_1+1,b_2-1,0,0)  =  (-c,-c,b-1,b+1) \\
\\
\la_2 (-c_1,-c_2,b_1-1,b_2+1)+ \mu_2 (b_1+1,b_2-1,0,0)  =  (\ol b+1,\ol b-1,a,a)
\end{array} \right.
\end{equation}

From the second equation, if $a \not=0$ then $b_1-1=b_2+1$. Using
the first equation this implies that $b-1=b+1$, which is not
possible. Hence, $a=0$ and consequently $\la_2=0$. From the
condition on the parameters $a,b,c$ (cf. eq.(\ref{BC3})), we also
have:
\begin{equation*}
(1+\ol b)(1-b) -ac \not=0 \Longrightarrow b \not=\pm 1 \, .
\end{equation*}

The system (\ref{S1}) then reduces to the simpler systems:
\begin{equation} \label{S2}
\left\{ \begin{array}{lll} \la_1 c_1- \mu_1 (b_1+1) & = & c \\
\la_1 c_2- \mu_1 (b_2-1) & = & c
\end{array} \right.
\end{equation}
and
\begin{equation*}
\left\{ \begin{array}{lll} \la_1 (b_1-1) & = & b-1 \\
\la_1 (b_2+1) & = & b+1
\end{array} \right. \qquad \wedge \qquad
\left\{ \begin{array}{lll} \mu_2 (b_1+1) & = & \ol b+1 \\
\mu_2 (b_2-1) & = & \ol b-1
\end{array} \right.
\end{equation*}
Since $b\not=\pm 1$, the two latter systems imply
$$
\la_1, \mu_2 \not=0 \quad , \quad b_1,b_2 \not=\pm 1 \, .
$$
Adding and subtracting the two equations in the two latter
systems:
\begin{equation} \label{S4}
\left\{ \begin{array}{lll} b &=& {1 \over 2} \la_1 (b_1+b_2) \\
\ol b &=& {1\over 2} \mu_2 (b_1+b_2)
\end{array} \right. \qquad \wedge \qquad
\left\{ \begin{array}{lll} \la_1 (b_2-b_1 +2) & = & 2 \\
\mu_2 (b_2-b_1 -2) & = & -2
\end{array} \right.
\end{equation}
Hence, $b_2-b_1 \not= \pm 2 $ and
\begin{equation} \label{S5}
b= \frac{b_1+b_2}{b_2-b_1 +2} \qquad , \qquad \ol b = -\frac{b_1+b_2}{b_2-b_1 -2} \, .
\end{equation}
Then
$$
\frac{b_1+b_2}{b_2-b_1 +2} = -\frac{\ol b_1+\ol b_2}{\ol b_2-\ol
b_1 -2} \quad \Longleftrightarrow \quad
\left\{ \begin{array}{lll} b_2 \ol b_2 & = & b_1 \ol b_1  \\
\ol b_1+ \ol b_2 & = & b_1+b_2
\end{array} \right.
$$
and so:
\begin{equation} \label{S6}
b_1=\ol b_2 \qquad \vee \qquad b_1=-b_2 \, .
\end{equation}
In the first case, from (\ref{S5}) and (\ref{S4}):
\begin{equation} \label{S7}
b= \frac{b_1+\ol b_1}{\ol b_1-b_1 +2} \quad , \quad \la_1= \frac{2}{\ol b_1-b_1 +2} \quad , \quad \mu_2=
\frac{-2}{\ol b_1-b_1 -2}
\end{equation}
and in the second:
\begin{equation} \label{S8}
b= 0 \quad , \quad \la_1= \frac{1}{-b_1+1} \quad , \quad \mu_2=
\frac{1}{b_1+1} \, .
\end{equation}
The previous formulas (\ref{S6}), (\ref{S7}) and (\ref{S8}) almost complete the proof of the statements (1a) and
(1b). It remains only to proof that, in the two cases (\ref{S6}), for arbitrary $(c_1,c_2)$ there exists $\mu_1$
and $c$ such that (\ref{S2}) holds.

We start by considering the first case $b_1= \ol b_2$. Subtracting
the two equations in (\ref{S2}), and taking into account
(\ref{S7}), we get:
$$
\mu_1= \frac{2(c_2-c_1)}{(\ol b_1-b_1)^2 -4} \, .
$$
Replacing this result in (\ref{S2}), we find
\begin{equation*}
c= \frac{2c_1(\ol b_1-1) -2c_2(b_1+1)}{(\ol b_1-b_1)^2 -4} \, .
\end{equation*}
Since $c$ is real, the triple $(c_1,c_2,b_1)$ has to satisfy the condition
\begin{equation*}
Im \left[ c_1(\ol b_1-1) -c_2(b_1+1) \right] =0 \, .
\end{equation*}
This concludes the proof of statement (1a) in the Theorem.

We proceed by considering the second case in (\ref{S6}):
$b_1+b_2=0$ and $b_1 \not= \pm 1$. Adding and subtracting the two
equations in (\ref{S2}), and taking into account (\ref{S8}), we
find
\begin{equation*}
\mu_1= \frac{c_2-c_1}{2(b_1^2-1)} \quad , \quad c= \frac{c_1 +c_2}{2(1-b_1)}
\end{equation*}
and, since $c$ is real, $c_1, c_2$ should satisfy the condition
\begin{equation*}
Im \left[ \frac{c_1 +c_2}{2(1-b_1)} \right] =0 \, .
\end{equation*}
This concludes the proof of the statement (1b) of the Theorem.\\

It remains to prove the statement (2). This basically amounts to invert the equations (\ref{ABC1}) and
(\ref{ABC2}) (in the two cases (1a) and (1b), respectively) and determine for which values of $(a,b,c)$ is this
possible.

We start by considering the case (1a): $b_1=\ol b_2$, $b_1
\not=\pm 1$. It follows from (\ref{S5}) that
$$
(\ol b_1-b_1+2)b = (b_1- \ol b_1+2)\ol b \, \Longleftrightarrow \, (\ol b_1-b_1)(b + \ol b)=2(\ol b - b) \, .
$$
Then we have three possibilities:

(i) If $b + \ol b=0$ then also $\ol b-b=0$ and so $b=0$. Hence,
$b$ cannot be pure imaginary number.

(ii) If $b=0$ then, from (\ref{S4}), it follows that $b_1+\ol b_1=0$, and so $b_1$ is an arbitrary imaginary
number. Hence, $b_2= \ol b_1 = -b_1$, and this is the case (1a), which will be consider below.

(iii) Finally, if $b +\ol b \not=0$ then $\ol b_1-b_1= 2{\ol b -b
\over b + \ol b}$. Substituting in (\ref{S7}) (which is valid in
the first case $b_1 = \ol b_2$), we get $\ol b_1+b_1= 4{\ol b b
\over b + \ol b}$. Then
\begin{equation}\label{S14}
b_1= \frac{2b \ol b + b - \ol b}{b + \ol b}
\end{equation}
which proves (\ref{abc2}).

We then consider the equation for $c$ in (\ref{ABC1}). Using (\ref{S14}), we get
\begin{equation} \label{S15}
c  =  \frac{b + \ol b}{2} \left[\frac{c_2 - c_1}{2} + \frac{c_1 +
c_2}{2 \ol b} \right] = c_1 X_1 + c_2 X_2
\end{equation}
where $X_1,X_2$ are given by (\ref{abc5}). Since $b + \ol b
\not=0$ and $b \not=\pm 1$, then also $X_1,X_2 \not=0$. Hence,
given $b$ and $c$, we immediately realize that the solutions
$(c_1,c_2)$ of (\ref{S15}) can be written in the form
(\ref{abc4}), which concludes the proof of (2a).

We proceed with the study of the case (1b): $b_1+b_2=0$, $b_1
\not=\pm 1$. This always yields $b=0$. It follows immediately from
(\ref{ABC2}) that, for an arbitrary $c \in \RE$, the operator $\wh
L^I_{(0,0,c)}$ admits a boundary potential representation with
$b_1$ an arbitrary number in $\CO \backslash \{-1,1\}$, $b_2=-b_1$
and $c_1,c_2 \in \CO$ such that:
\begin{equation}
c_1+c_2=2 c (1-b_1) \, .
\end{equation}
This proves (2b).

\end{proof}

If the parameters of the boundary potential $\wh B$ are real, the
previous result considerably simplifies:

\begin{corollary} \label{CorrInt}
Let $\wh H=\wh H_0 + \wh B$ where $\wh B$ is given by (\ref{B1})
with $(c_1,c_2,b_1,b_2) \in \RE^4$. Then $\wh H =\wh
L^I_{(a,b,c)}$ for some $a,c \in \RE$ and $b \in \CO$ iff:
$$
b_1 = \pm b_2 \quad , \quad b_1 \not= \pm 1 \quad , \quad a=0
\quad , \quad b=\tfrac{b_1+b_2}{2}
$$
In addition:

(i) If $b_1=b_2$ then $c=\tfrac{c_1(1-b_1) +c_2(1+b_1)}{2}$,

(ii) If $b_1=-b_2$ then $c=\tfrac{c_1+c_2}{2(1-b_1)}$.

\end{corollary}

\begin{proof}
The proof is a direct consequence of (1a) and (1b) in Theorem
\ref{TheoIn}.

\end{proof}

We proceed with the separating case:

\begin{theorem}\label{TheoSe}
Let $(a_-,b_-),(a_+,b_+) \in \RE^2 \backslash \{(0,0)\}$. Then $L^S_{(a_-,a_+,b_-,b_+)}= \wh H_0 + \wh B$ for
some boundary potential operator of the form (\ref{B1}) iff one of the following holds true:\\
\\
(1a) $a_-=0$, $a_+ \not=0$, $b_-\not=0$. In this case $b_1=b_2=1$, $c_1$ is arbitrary and $c_2=  2 b_+/a_+ $.
This corresponds to Dirichlet boundary conditions at $x=0^-$ and Robin (or Neumann, if $b_+=0$) boundary conditions at $x=0^+$.\\
\\
(1b) $a_-\not=0$, $a_+ =0$, $b_+ \not= 0$. In this case $b_1=b_2=-1$, $c_1=-2 b_-/a_-$. This corresponds
to Dirichlet boundary conditions at $x=0^+$ and Robin (or Neumann, if $b_-=0$) boundary conditions at $x=0^-$.\\
\\
(1c) $a_-=a_+=0$ and $b_-,b_+ \not=0$. In this case $b_1=1$, $b_2=-1$ and $c_1+c_2 \not=0$. This corresponds to
Dirichlet boundary conditions at $x=0^-$ and $x=0^+$.

\end{theorem}

\begin{proof}
The two sets of boundary conditions (\ref{BC1}) and (\ref{BC2})
are equivalent iff exists $\la_1, \la_2, \mu_1,\mu_2 \in \CO$ such
that:
\begin{equation} \label{D1}
\left\{ \begin{array}{lll} \la_1 (-c_1,-c_2,b_1-1,b_2+1)+ \mu_1 (b_1+1,b_2-1,0,0) & = & (b_-,0,-a_-,0) \\
\la_2 (-c_1,-c_2,b_1-1,b_2+1)+ \mu_2 (b_1+1,b_2-1,0,0) & = & (0,b_+,0,-a_+)
\end{array} \right.
\end{equation}
This implies:

\begin{equation} \label{D2}
\left\{ \begin{array}{lll} \la_1 (b_2+1) & = & 0 \\
\la_1 (b_1-1) & = & -a_-
\end{array} \right. \qquad \wedge \qquad
\left\{ \begin{array}{lll} \la_2 (b_1-1) & = & 0 \\
\la_2 (b_2+1) & = & -a_+
\end{array} \right.
\end{equation}
From these systems it follows that:
$$
a_+ \not= 0 \Longrightarrow b_2+1 \not=0 \Longrightarrow \la_1=0
\Longrightarrow a_-=0
$$
and likewise $a_- \not=0 \Longrightarrow a_+=0$. Hence,
\begin{equation*}
a_+=0 \quad \vee \quad a_-=0 \, .
\end{equation*}
We then have three distinct cases:\\
\\
(i) $a_-=0  \wedge a_+ \not=0$. Then also $b_- \not= 0$. From
(\ref{D2}), $\la_2\not=0$, $b_2+1 \not=0$ and so $\la_1=0$ and
$b_1=1$. Hence, from (\ref{D1}),
$$
\mu_1(b_1+1,b_2-1)=(b_-,0)
$$
and since $b_- \not=0$, we have $\mu_1 \not=0$ and so $b_2=1$.
Substituting in (\ref{D2}), we also find $\la_2=-a_+/2$.

From (\ref{D1}), we then have
$$
-\la_2 c_2 + \mu_2 (b_2-1)=b_+ \Longrightarrow \la_2 c_2 = -b_+ \Longrightarrow c_2=2 b_+ /a_+ \, .
$$
This concludes the proof of the case (1a).\\
\\
(ii) $a_-\not=0 \wedge a_+ =0$. This case yields the conditions (1b). The proof follows exactly the same steps as in (i).\\
\\
(iii) $a_+=a_-=0$ and $b_-,b_+ \not=0$. From (\ref{D2}), we get
$\la_1=\la_2=0$ or $(b_1,b_2)=(1,-1)$. If $\la_1=\la_2=0$ then,
from (\ref{D1}),
$$
\mu_1 (b_1+1,b_2-1)=(b_-,0)  \qquad \wedge \qquad \mu_2
(b_2-1)=b_+
$$
and the first equation implies $\mu_1 \not=0$ and $b_2 =1$, while
the second yields $b_2 \not=1$. Hence, $(\la_1,\la_2) \not=(0,0)$
and so $b_1=1$ and $b_2=-1$. Then, from (\ref{D1})
$$
\left\{ \begin{array}{lll} 2\mu_1 - \la_1 c_1 & = & b_- \\
-2 \mu_1 - \la_1 c_2  & = & 0
\end{array} \right.
\quad  \Longrightarrow \quad c_1+c_2 \not=0
$$
which concludes the proof of the case (1c).

\end{proof}

We now determine the sesquilinear form associated with $\wh H$.

\begin{theorem}\label{SQF}
Let $\wh H=\wh H_0+\wh B$, where $\wh B$ is the boundary potential
given by (\ref{B1}). The sesquilinear form associated with $\wh H$
is given in the domain $\DO(h)=\H^2(\RE_-) \oplus \H^2(\RE_+)$ by
\begin{equation}\label{BF1}
h(\psi,\phi)=\left(\psi'_-,\phi_-'\right)_{\L^2(\RE_-)}+\left(\psi'_+,\phi_+'\right)_{\L^2(\RE_+)}+
b_B(\psi,\phi)
\end{equation}
where the boundary term is
\begin{equation}\label{BF3}
b_B(\psi,\phi)=\ol\phi_+(0) \psi_+'(0)-\ol\phi_-(0) \psi_-'(0) \,
.
\end{equation}

If $\wh H$ is s.a. then $h$ can be extended to $\DO(h)=\H^1(\RE_-)
\oplus \H^1(\RE_+)$. In this case, $h$ is also given by
(\ref{BF1}), but the boundary term $b_B$ depends on the particular
operator $\wh H$:\\
\\
(1) If $\wh H$ is a s.a interacting Schr\"odinger operator then
the parameters defining $\wh B$ satisfy (cf. Theorem \ref{TheoIn})
$b_1 \not=\pm 1$ and $b_1 = \ol b_2$ or $b_1+b_2=0$. In this case:
\begin{equation} \label{BFI}
b_B(\psi,\phi)= \left\{ \begin{array}{lll}
\tfrac{c_1}{1-b_1}\ol\phi_-(0) \psi_-(0) + \tfrac{c_2}{1+\ol b_1}
\ol\phi_+(0)
\psi_+(0) & , & b_1= \ol b_2 \\
\\
\tfrac{c_1}{1-b_1}\ol\phi_-(0) \psi_-(0) + \tfrac{c_2}{1- b_1}
\ol\phi_+(0) \psi_+(0) & , & b_1= - b_2
\end{array} \right.
\end{equation}
\\
(2) If $\wh H$ is a s.a. separating Schr\"odinger operator then
the parameters of $\wh B$ satisfy (cf. Theorem \ref{TheoSe})
$b_1=b_2=\pm 1$ or $b_1=-b_2=1$ and $c_1+c_2 \not=0$. In this
case:
\begin{equation} \label{BFS}
b_B(\psi,\phi)= \left\{ \begin{array}{lll}  \tfrac{c_2}{2}
\ol\phi_+(0)
\psi_+(0) & , & b_1= b_2=1 \\
0 & , & b_1=-  b_2= 1 \, , \,  c_1 + c_2 \not=0\\
\tfrac{c_1}{2}\ol\phi_-(0) \psi_-(0)  & , & b_1= b_2=-1
\end{array} \right.
\end{equation}

\end{theorem}

\begin{proof}

The sesquilinear form generated by $\wh H$ is by definition
$$
h(\psi,\phi)= \left(\wh H\psi,\phi\right)_{\L^2} \quad , \quad
\DO(h)=\DO_\max(\wh H)
$$
where $\left(\, ,\, \right)_{\L^2}$ is the standard inner product
in $\L^2$. Since $\wh H \subseteq \wh S^*$
$$
h(\psi,\phi)=\left(\wh
S^*\psi,\phi\right)_{\L^2}=-\left(\psi''_-,\phi_-\right)_{\L^2(\RE_-)}-\left(\psi''_+,\phi_+\right)_{\L^2(\RE_+)}
$$
for all $\psi,\phi \in \DO_\max(\wh H)$. Integrating by parts, we
immediately obtain (\ref{BF1}) and (\ref{BF3}). This expression is
well-defined in $D(h)=\H^2(\RE_-) \oplus \H^2(\RE_+)$.

We now consider the case where $\wh H$ is a s.a interacting
Schr\"odinger operator. In this case, the parameters in $\wh B$
satisfy (cf. Theorem \ref{TheoIn}) $b \not= \pm 1$. From
(\ref{BF3}) and using the boundary conditions (\ref{BC1}), we get:
\begin{equation} \label{BF2}
b_B(\psi,\phi)=\frac{\ol\phi_-(0)}{1-\ol b_2} \left( (1+ \ol b_1)
\psi_+'(0) +(\ol b_2 -1)  \psi_-'(0) \right) \, .
\end{equation}
We now set $b_1= \ol b_2$ (the case (1a) in Theorem \ref{TheoIn}).
Using again the boundary conditions (\ref{BC1}), we obtain from
(\ref{BF2})
$$
b_B(\psi,\phi)=\frac{\ol\phi_-(0)}{1-b_1} \left( c_1 \psi_-(0)
+c_2   \psi_+(0) \right) \, .
$$
Finally, using again (\ref{BC1}), we obtain the first formula in
(\ref{BFI}).

The other case in Theorem \ref{TheoIn} is $b_1=-b_2$. Substituting
in (\ref{BF2}) and proceeding as in the first case, we obtain the
second formula in (\ref{BFI}).

Finally, we consider the case where $\wh H$ is a separating
Schr\"odinger operator. Then the parameters in $\wh B$ satisfy
(cf. Theorem \ref{TheoSe}) $b_1=b_2=\pm 1$ or $b_1=-b_2=1$ and
$c_1+c_2 \not=0$.

If $b_1=b_2=1$ then from Theorem \ref{TheoSe} we have Dirichlet
boundary conditions at $x=0^-$, and Robin boundary conditions at
$x=0^+$ (the latter are of the form $\psi'_+(0)=c_2 \psi_+(0)/2$).
Substituting these conditions in (\ref{BF3}), we obtain the first
formula in (\ref{BFS}). The other two cases in (\ref{BFS}) are
proved exactly in the same way.

\end{proof}

\end{subsection}

\begin{subsection}{The $\delta$ and $\delta'$ potentials: Norm
resolvent versus boundary operator formulations}

We now consider the case of a $\delta$ and a $\delta'$ potential
operator and a $\delta'$ pseudo potential operator and determine
the explicit form of the corresponding Schr\"odinger operators
(Corollaries \ref{CorrD}, \ref{CorrD'1} and \ref{CorrD'2}). We
will see that these Schr\"odinger operators are closely related to
the norm resolvent limit operators (\ref{K}) and (\ref{K'}).

Let $B=b \delta'$. A family of boundary potential operators associated with $B$ is
\begin{equation}\label{Delta'1}
\widehat B_{\delta',1} \psi = c \delta'(x) * \psi + d \psi *
\delta'(x)= (c \widehat\delta'_+ +d \widehat\delta'_-) \psi \quad
, \quad c+d=b
\end{equation}
which is well-defined for all $\psi \in \A^1$ and satisfies
$$
\widehat B_{\delta',1} \psi = b\delta'(x) \cdot \psi \quad , \quad
\psi \in \C^1.
$$
Hence, the operators $\widehat B_{\delta',1}$ are extensions of
$b\delta'(x) \cdot$ to the space $\A^1 \supset \C^1$. As usual,
there are many possible extensions. A more general family of
boundary operators associated with $ B=(c+d) \delta'(x)$ is
\begin{equation}\label{Delta'2}
\widehat B_{\delta',2}  =c \widehat\delta'_+  + d
\widehat\delta'_- + eD_x (\widehat\delta_+ -\widehat\delta_- ) + f
(\widehat\delta_+ -\widehat\delta_- ) D_x \quad , \quad c,d,e,f
\in \RE \, .
\end{equation}
These are boundary {\it pseudo potential} operators, all of which
are also well-defined on $\A^1$ and satisfy $\widehat
B_{\delta',2} \psi= (c+d) \delta'(x) \cdot \psi$ for all $\psi \in
\C^1$.

Finally, let us introduce the operators
\begin{equation}
\widehat B_{\delta}=c \widehat\delta_+  + d  \widehat\delta_- \, , \label{Delta}
\end{equation}
which are boundary potential operators associated with the
$\delta$-potential $B= (c+d) \delta(x)$.

The boundary operators $\widehat B_\delta=c \widehat\delta_+  + d
\widehat\delta_- $ and $\widehat B_{\delta',1}=c \widehat\delta_+
' + d \widehat\delta_- '$ are particular examples of operators of
the form (\ref{B1}). We then have

\begin{corollary}\label{CorrD}
The operator $\widehat H=-D^2_x + \widehat B_\delta$ is given explicitly by (\ref{K}) with $a=c+d$.
\end{corollary}

\begin{proof}
The boundary operator $\wh B_{\delta}$ is of the form (\ref{B1})
with $c_1=d$, $c_2=c$ and $b_1=b_2=0$. Hence, from Theorem
\ref{CorrH}, $\wh H \subseteq \wh S^*$ and $\DO_\max(\wh H)$ is
characterized by the two boundary conditions (\ref{BC1}):
$$
 \left\{ \begin{array}{l}
-d \psi_-(0) -c \psi_+(0)  - \psi_-'(0) + \psi_+'(0) =0 \\
\psi_-(0)-\psi_+(0)=0
\end{array} \right.
$$
We conclude that $\wh H$ is exactly the operator (\ref{K}) with
$a=c+d$.
\end{proof}

The $\delta'$-potential is more subtle. We have

\begin{corollary}\label{CorrD'1}
Let $\widehat H=-D^2_x +\widehat B_{\delta',1}$ where $\widehat B_{\delta',1}=c
\widehat\delta_+ ' + d \widehat\delta_- '$, $c,d \in \RE$. Then

(1) For $c=d\not= 1$ the operator $\widehat H$ is given explicitly by (\ref{K'}), with $\theta=\frac{c+1}{1-c}$.

(2) Conversely, all operators $\widehat H_{b\delta',\theta}$,
$\theta \not=-1$ given by (\ref{K'}) can be written in the form
$\widehat H=-D^2_x +\widehat B_{\delta',1}$, with
$c=d=\frac{\theta-1}{\theta+1}$.

(3) The operator $\widehat H$ is of the form (\ref{K'}) only for $c=d\not=1$ and for $c=-d \notin \{-1,1\}$. The family
$c=-d \in \RE \backslash \{-1,1\}$ generates the single operator $\widehat H_{b\delta',\theta}$ with $\theta =1$.
\end{corollary}

\begin{proof}
(1) Since $\widehat B_{\delta',1}$ is of the form (\ref{B1}) with
$c_1=c_2=0$ and $b_1=d$, $b_2=c$, it follows from Theorem
\ref{CorrH} that $\widehat H \subseteq \widehat S^*$ and that
$\DO_\max(\wh H)$ satisfies the boundary conditions (\ref{BC1}):
\begin{equation}\label{BoundaryConditionsK'}
\left\{ \begin{array}{l}
(d-1)\psi_-'(0)+(c+1)\psi_+'(0)=0 \\
\\
(d+1)\psi_-(0) +(c-1) \psi_+(0)=0
\end{array} \right.
\end{equation}
For $c=d\not=1$ we recover the boundary conditions that characterize the domain of the operators $\widehat
H_{b\delta',\theta}$ (given by (\ref{K'}))
$$
\psi_+(0)=\theta \psi_-(0) \quad \mbox{and} \quad \theta \psi_+'(0)
= \psi_-'(0) \quad , \quad \theta \in \RE \backslash \{-1\}
$$
with $\theta =\frac{c+1}{1-c}$.

(2) All the values $\theta \in \RE\backslash \{-1\}$ can be obtained from suitable values of $c$
$$
\theta =\frac{c+1}{1-c} \Longleftrightarrow c =\frac{\theta-1}{\theta+1}.
$$
The case $\theta =-1$ cannot be generated by the family of boundary operators $\widehat B_{\delta',1}$,
even if we consider the more general case $c\not= d$.
This can be easily realized from equation (\ref{BoundaryConditionsK'}).

(3) For $c=-d \in \RE \backslash \{-1,1\}$, eq.(\ref{BoundaryConditionsK'}) yields the boundary conditions that
correspond to the case $\theta=1$.
Finally, when $c=1$ or $d=1$, the boundary conditions (\ref{BoundaryConditionsK'}) are obviously not of the form
(\ref{K'}) and when $c\not=\pm d$ (for $c,d \not=1$) the
operators $-D^2_x +\widehat B_{\delta',1}$ are also {\it not} of the form (\ref{K'}) since
$$
\frac{d+1}{1-c}=\frac{1+c}{1-d} \Longleftrightarrow c=\pm d \, .
$$
\end{proof}

We conclude that for each value of $b=c+d$ the family of operators
$\widehat H=-D^2_x +\widehat B_{\delta',1}$ generates one, and
only one, operator of the form (\ref{K'}). Hence, there is no
hidden structure in the $\delta'$-potential operators $\widehat
B_{\delta',1}$ that may affect the modelling of the operators
$\widehat H_{b\delta',\theta}$. This situation changes when we
consider the more general family $\widehat B_{\delta',2}$. The
extra terms, which yield a zero contribution when acting on smooth
functions, generate for each value of $c+d$ the entire family of
operators (\ref{K'}). This is proved in the next corollary.

\begin{corollary}\label{CorrD'2}
Let $\widehat H=-D^2_x +\widehat B_{\delta',2}$ where $\widehat B_{\delta',2}$ is given by eq.(\ref{Delta'2})
$$
\widehat B_{\delta',2}  =c \widehat\delta'_+  + d
\widehat\delta'_- + eD_x (\widehat\delta_+ -\widehat\delta_- ) + f
(\widehat\delta_+ -\widehat\delta_- ) D_x \quad, \quad c,d,e,f \in
\RE
$$
We then have:

(1) For $c=d$ and $e=f \not=1-c$, the operator $\widehat H$ is given explicitly by (\ref{K'}) with
$\theta=\frac{e-1-c}{e-1+c}$.

(2) Each family of operators of the form $\widehat H=-D^2_x +\widehat B_{\delta',2}$, obtained by fixing the value of $c=d \not=0$ and
letting $e=f \in \RE\backslash \{1-c\}$, contains all the operators (\ref{K'}) with $\theta \in \RE\backslash \{1\}$.

(3) All the operators $\widehat H=-D^2_x +\widehat B_{\delta',2}$ with $c=d=0$ and $e=f \in \RE\backslash \{1\}$
are given explicitly by the operator $\widehat H_{b\delta',\theta}$ with $\theta=1$.

\end{corollary}

\begin{proof}
Since supp $\widehat B_{\delta',2} \psi \subseteq \{0\}$ for all
$\psi \in \DO(\widehat B_{\delta',2})$, it follows from Theorem
\ref{TheoA} that $\wh H \subseteq \wh S^*$. Moreover, also from
Theorem \ref{TheoA}, $\psi \in \DO_\max(\widehat H)$ iff
$$
\psi \in \H^2(\RE_-) \oplus \H^2(\RE_+) \quad \mbox{and} \quad \widehat F_{\delta',2}\psi=0 \, .
$$
The operator $\widehat F_{\delta',2}$ is
$$
\widehat F_{\delta',2}  =  -2 \widehat\beta D_x - \widehat\beta' +c \widehat\delta'_+ + d \widehat\delta'_-
+e D_x \widehat\beta + f  \widehat\beta D_x
$$
and so,
\begin{equation}
\begin{array}{l}
\widehat F_{\delta',2} \psi  = 0\\
\Longleftrightarrow D_x \left[\delta(x) (c \psi_++ d \psi_-
+(e-1)(\psi_+-\psi_-)) \right] \\
\qquad \quad - \delta(x) \left[ c \psi_+' +d \psi_-' - (f-1)(\psi'_+-\psi_-')
\right] = 0 \\
\Longleftrightarrow \left\{ \begin{array}{l}
(1-c-e)\psi_+(0)=(1+d-e)\psi_-(0) \\
\\
(1+c-f)\psi_+'(0) =(1-d-f) \psi_-'(0)
\end{array} \right.
\end{array}
\label{BoundaryConditionsK'2}
\end{equation}
For $c=d$ and $e=f$ we get
$$
 \left\{ \begin{array}{l}
(1-e-c)\psi_+(0)=(1-e+c)\psi_-(0) \\
\\
(1-e+c)\psi_+'(0) =(1-e-c) \psi_-'(0)
\end{array} \right.
$$
For $e\not=1-c$ we recover the boundary conditions for the operators $\widehat H_{b\delta',\theta}$
with $\theta =\frac{1-e+c}{1-e-c}$. This proves (1).

(2) This statement is easily proved by noticing that, for arbitrary fixed $c \not=0$ and $\theta \not=1$,
$$
\theta =\frac{1-e+c}{1-e-c} \Longleftrightarrow e =\frac{\theta(c-1)+1+c}{1-\theta}
$$
and that the solution of the equation satisfies $e\not=1-c$.

(3) It follows directly from (\ref{BoundaryConditionsK'2}) that for $c=d=0$ and $e=f$, we have $\theta
=\frac{1-e}{1-e}=1$, for all values of $e \not=1$.
\end{proof}

Finally, we remark that a boundary pseudo potential operator
formulation of the $\delta$-potential (of the form (\ref{Delta'2})
but for the $\delta$-potential) yields exactly the Schr\"odinger
operators that were determined in Corollary \ref{CorrD}. This can
be easily realized by reproducing the calculations of Corollary
\ref{CorrD'2} for the new boundary operators. Hence, this more
general formulation of the $\delta$-potential does not display an
"inner structure" as the one of the $\delta'$-potential. This
agrees with the results of the norm resolvent limit formulation
(\ref{K}).

\bigskip

\noindent\textbf{Acknowledgements}. The authors would like to
thank the anonymous referee for several insights and useful
suggestions.

Nuno Costa Dias and Jo\~{a}o Nuno Prata have been supported by the
research grant PTDC/MAT-CAL/4334/2014 of the Portuguese Science
Foundation.

Cristina Jorge was supported by the PhD grant SFRH/BD/85839/2012
of the Portuguese Science Foundation.

\end{subsection}

\end{section}

\vspace{2cm}

*******************************************************************

\textbf{Author's addresses:}

\begin{itemize}
\item \textbf{Nuno Costa Dias} and \textbf{Cristina Jorge}:
Departamento de Matem\'{a}tica. Universidade Lus\'{o}fona de
Humanidades e Tecnologias. Av. Campo Grande, 376, 1749-024 Lisboa,
Portugal and Grupo de F\'{\i}sica Matem\'{a}tica, Universidade de
Lisboa, Av. Prof. Gama Pinto 2, 1649-003 Lisboa, Portugal

\item \textbf{Jo\~ao Nuno Prata: }Escola Superior N\'autica
Infante D. Henrique. Av. Eng. Bonneville Franco, 2770-058 Pa\c{c}o
d'Arcos, Portugal and Grupo de F\'{\i}sica Matem\'{a}tica,
Universidade de Lisboa, Av. Prof. Gama Pinto 2, 1649-003 Lisboa,
Portugal

\end{itemize}

\email{(NCD) ncdias@meo.pt}

\email{(CJ) cristina.goncalves.jorge@gmail.com}

\email{(JNP) joao.prata@mail.telepac.pt}

*******************************************************************


\begin{thebibliography}{99}


\bibitem{Albeverio0} S. Albeverio, C. Cacciapuoti, D. Finco: Coupling in the singular limit of thin quantum waveguides, {\it J. Math. Phys.} {\bf 48} (3) (2007), 032103, 21pp.

\bibitem{Albeverio1} S. Albeverio, F. Gesztesy, R. H\"ogh-Krohn, H. Holden: {\it Solvable Models in Quantum Mechanics}, 2nd ed., (AMS, Chelsea, 2005).

\bibitem{Albeverio7} S. Albeverio, F. Gesztesy, R. H\"ogh-Krohn: The low energy expansion in nonrelativistic
scattering theory, {\it Ann. Inst. H. P.} {\bf A37} (1982) 1--28.

\bibitem{Albeverio2} S. Albeverio, P. Kurasov, {\it Singular perturbations of differential operators and solvable Schr\"odinger type operators},
(Cambridge University Press, 2000).

\bibitem{Albeverio3} S. Albeverio, V. Koshmanenko: Singular rank one perturbations of selfadjoint operators and Krein theory of selfadjoint
extensions, {\it Potential Anal.} {\bf 11} (1999) 279-287.

\bibitem{Albeverio4} S. Albeverio, L. Nizhnik: Approximation of general zero-range potentials, {\it Ukraimian Math. J.} {\bf 52} no. 5 (2000) 664-672.

\bibitem{Albeverio5} S. Albeverio, R. H\"ogh-Krohn: Point interactions as limits of short range interactions, {\it J. Operator Theory}
{\bf 6} (1981) 313-339.

\bibitem{Albeverio6} S. Albeverio, F. Gesztesy, R. H\"ogh-Krohn, W. Kirsch: On point interactions in one dimension, {\it J. Operator Theory}
{\bf 12} (1984) 101-126.

\bibitem{Antonevich} A. Antonevich: The Schr\"odinger equation with point interactions in an algebra of new generalized functions. In: {\it
Nonlinear theory of generalized functions} Chapman and Hall, {\it Research notes in mathematics series}, {\bf 401} (1999).

\bibitem{Avron} J. Avron, P. Exner, Y. Last: Periodic Schr\"odinger-operators with large gaps and Wannier-Stark ladders, {\it Phys. Rev. Lett.} {\bf 72} (1994) 896--899.

\bibitem{Berezin} F.A. Berezin, L.D. Fadeev: Remark on the Schr\"odinger equation with singular potential, {\it Dokl. Akad. Nauk. SSSR} {\bf 137} (1961) 1011.

\bibitem{Brasche} J.F. Brasche, R. Figari, A. Teta: Singular Schr\"odinger operators as limits of point interaction Hamiltonians, {\it Potential
Analysis} {\bf 8} (1998) 163-178.

\bibitem{Bolle1} D. Boll\'e, F. Gesztesy, M. Klaus: Scattering theory for one-dimensional systems with $\int dx
V(x) =0$, {\it J. Math. Anal. Appl.} {\bf 122} (1987) 496--518.

\bibitem{Bolle2} D. Boll\'e, F. Gesztesy, S. Wilk: A complete treatment of low-energy scattering in one
dimension, {\it J. Operator Theory} {\bf 13} (1985) 3--31.

\bibitem{Cacciapuoti1} C. Cacciapuoti, P. Exner: Nontrivial edge coupling from a Dirichlet network squeezing: the case
of a bent waveguide, {\it J. Phys. A: Math. Theor.} {\bf 40} (26)
(2007) F511-F523.

\bibitem{Cacciapuoti2} C. Cacciapuoti, D. Finco: Graph-like models for thin waveguides with robin boundary conditions, {\it Asymptotic Analysis} {\bf 70} (3-4) (2010) 199-230.

\bibitem{Cheon} T. Cheon, P. Exner, P. Seba: Wave function shredding by sparse quantum barriers, {\it Phys. Lett. A} {\bf 277} (2000) 1--6.

\bibitem{Christiansen} P. Christiansen, H. Arnbak, A. Zolotaryuk, V. Ermakov, Y. Gaididei: On the existence of
resonances in the transmission probability for interactions arising from derivatives of the Dirac delta
function, {\it J. Phys. A: Math. Gen} {\bf 36} (2003) 7589--7600.

\bibitem{Colombeau} J.F. Colombeau, {\it New generalized functions and multiplication of distributions}, (North Holland, 1989).

\bibitem{Antonio} G. Dell'Antonio, G. Panati: The flux-across-surfaces theorem and zero-energy resonances, {\it
J. Stat. Phys.} {\bf 116} (2004) 1161--1180.

\bibitem{Dias1} N.C. Dias, J.N. Prata: A multiplicative product of distributions and a
class of ordinary differential equations with distributional
coefficients, {\it J. Math. Anal. Appl.} {\bf 359} (2009) 216-228.

\bibitem{Dias5} N.C. Dias, J.N. Prata: Wigner functions with boundaries, {\it J. Math. Phys.} {\bf 43} (2002)
4602--4627.

\bibitem{Dias4} N.C. Dias, J.N. Prata: Deformation quantization of confined systems, {\it Int. J. Quant. Inf.}
{\bf 5} (2007) 257--263.

\bibitem{Dias3} N.C. Dias, A. Posilicano, J.N. Prata:
Self-adjoint, globally defined Hamiltonian operators for systems with boundaries, {\it Comm. Pure Appl. Anal.}
{\bf 10}, no.6 (2011) 1687-1706.

\bibitem{Exner2} P. Exner: Lattice Kronig-Penney models, {\it Phys. Rev. Lett.} {\bf 74} (1995) 3503-3506.

\bibitem{Exner} P. Exner, H. Neidhardt, V. Zagrebnov: Potential approximations to $\delta'$: an inverse Klauder phenomenon with norm resolvent
convergence, {\it Comm. Math. Phys.} {\bf 224} (2001) 593-612.

\bibitem{Garbaczewski} P. Garbaczewski, W. Karwowski: Impenetrable barriers and canonical quantization, {\it Am. J.
Phys.} {\bf 72} (2004) 924--933.

\bibitem{Golovaty1} Y.D. Golovaty, R.O. Hryniv: On norm resolvent convergence of Schr\"odinger operators with $\delta'$-like potentials,
{\it J. Phys. A: Math. Theor.} {\bf 43} (2010) 155204 (14pp).

\bibitem{Golovaty2} Y.D. Golovaty, S.S. Man'ko: Solvable models for the Schr\"odinger operators with $\delta'$-like potentials, {\it Ukr. Math.
Bulletin} {\bf 6} (2009) 173-207.

\bibitem{Golovaty3} Y.D. Golovaty, R.O. Hryniv: Norm resolvent convergence of sigularly scaled Schr\"odinger operators
and $\delta'$-potentials, {\it Proc. R. Soc. Edinb. A.} {\bf 143}, no.4 (2013) 791--816.

\bibitem{Grubb} G. Grubb: {\it Distributions and Operators}, Graduate Texts in Mathematics, 257 (Springer, 2009).

\bibitem{Hirokawa} M. Hirokawa, T. Kosaka: One-dimensional tunnel-junction formula for the Schr\"odinger
particle, {\it Siam J. Appl. Math} {\bf 73}, no.6, (2013) 2247--2261.

\bibitem{Hormander} L. H\"ormander, {\it The analysis of linear partial differential operators I} (Springer-Verlag, 1983).

\bibitem{Hormann} G. H\"ormann, L. Oparnica: Distributional solution concepts for the Euler-Bernoulli beam
equation with discontinuous coefficients, {\it Applicable Analysis} {\bf 86}, no.11 (2007) 1347--1363.

\bibitem{Jensen} A. Jensen, T. Kato: Spectral properties of Schr\"odinger operators and time-decay of the wave
functions, {\it Duke Mathematical Journal} {\bf 46} (1979) 583--611.

\bibitem{Kanwal} R.P. Kanwal, {\it Generalized Functions: Theory
and Technique}, Second Edition, (Birkh\"auser, 1998).

\bibitem{Koshmanenko} V.D. Koshmanenko, {\it Singular quadratic
forms in perturbation theory} (Kluwer Academic Publishers, 1999).

\bibitem{Kostenko} A. Kostenko, M. Malamud, 1-D Schr\"odinger
operators with local point interactions on a discrete set, {\it J.
Differential Equations} {\bf 249} (2010) 253-304.

\bibitem{Kurasov1} P. Kurasov: Distribution theory for discontinuous test functions and differential operators with generalized coefficients,
{\it J. Math. Anal. Appl} {\bf 201} (1996) 297-323.

\bibitem{Kurasov2} P. Kurasov, J. Boman: Finite rank singular perturbations and distributions with discontinuous test functions, {\it Proc. Amer. Math. Soc}
{\bf 126}, no.6 (1998) 1673-1683.

\bibitem{Kurasov3} P. Kurasov, A. Scrinzi, N. Elander: On the $\delta'$-interaction arrising in the exterior
complex scaling, {\it Phys. Rev. A} {\bf 49} (1994) 5095--5097.

\bibitem{Nizhnik2} L.P. Nizhnik: A one-dimensional Schr\"odinger operator with point interactions on Sobolev spaces, {\it Funct. Anal. Appl.} {\bf 40} (2006), N 2, 143-147.

\bibitem{Sarrico} C. Sarrico: Collision of delta-waves in a turbulent model studied via a distribution product,
{\it Nonlinear Anal.} {\bf 73}, no.9 (2010) 2868--2875.

\bibitem{Seba0} P. Seba: Schr\"odinger particle on a half line, {Lett. Math. Phys.} {\bf 10} (1985) 21--27.

\bibitem{Seba1} P. Seba: The generalized point interaction in one dimension, {\it Czech. J. Phys. B} {\bf 36} (1986), 667-673.

\bibitem{Seba2} P. Seba: Some remarks on the $\delta'$-interaction in one dimension, {\it Rep. Math. Phys.} {\bf 24} (1986) N 1, 111-120.

\bibitem{Zolotaryuk0} A.V. Zolotaryuk: Point interactions of the dipole type defined through a three-parametric
power regularization, {\it J. Phys A: Math. Theor.} {\bf 43} (2010) 105302 (21 pp).

\bibitem{Zolotaryuk2} A.V. Zolotaryuk: Boundary conditions for the states with resonant tunnelling across the $\delta'$-potential,
{\it Phys. Lett. A} {\bf 374} (15-16) (2010) 1636--1641.

\bibitem{Zolotaryuk1} A.V. Zolotaryuk, Y. Zolotaryuk: Controlling a resonant transmission across the $\delta'$-potential: the inverse problem, {\it J. Phys. A: Math and
Theor.} {\bf 44} (37) (2011) 375305, 21pp.



\end{thebibliography}
\end{document}